\documentclass{amsart}
\usepackage{amssymb,amsmath, amsfonts, amsrefs, tikz, epsfig, float}
\usepackage{amsthm}
\usepackage{mathrsfs}
\usepackage{enumitem, indentfirst}
\usepackage[fleqn,tbtags]{mathtools}
\usepackage{color}
 \newtheorem{theorem}{Theorem}[section]
 \newtheorem{corollary}[theorem]{Corollary}
 \newtheorem{lemma}[theorem]{Lemma}
 \newtheorem{proposition}[theorem]{Proposition}

 \theoremstyle{definition}

 \theoremstyle{remark}
 \newtheorem{remark}[theorem]{Remark}
  \numberwithin{equation}{section}

\renewcommand{\phi}{\varphi}
\renewcommand{\theta}{\vartheta}
\DeclareMathOperator{\sform}{\mathfrak{s}}
\DeclareMathOperator{\tform}{\mathfrak{t}}

\DeclareMathOperator{\wfrom}{\mathfrak{w}}

\DeclareMathOperator{\wform}{\mathfrak{w}}

\DeclareMathOperator{\ring}{\mathfrak{A}}

\DeclareMathOperator{\dist}{dist}

\DeclareMathOperator{\mul}{mul}
\DeclarePairedDelimiterX\sipt[2]{(}{)_{\tform}}{#1\,\delimsize\vert\,#2}
\DeclarePairedDelimiterX\sipv[2]{(}{)_{v}}{#1\,\delimsize\vert\,#2}
\DeclarePairedDelimiterX\sipw[2]{(}{)_{w}}{#1\,\delimsize\vert\,#2}

\newcommand{\Abs}[1]{\big\lvert#1\big\rvert}
\newcommand{\alg}{\mathscr{A}}
\newcommand{\balg}{\mathscr{B}}
\newcommand{\step}{\mathscr{S}}

\newcommand{\abs}[1]{\lvert#1\rvert}
\newcommand{\dupN}{\mathbb{N}}

\newcommand{\seq}[1]{(#1_{n})_{n\in\dupN}}

\newcommand{\nen}{n\in\mathbb{N}}
\newcommand{\dupR}{\mathbb{R}}
\newcommand{\dupC}{\mathbb{C}}

\newcommand{\dom}{\operatorname{dom}}

\newcommand{\ran}{\operatorname{ran}}

\newcommand{\lef}{\mathscr{L}(E;F)}

\newcommand{\bha}{\mathscr{B}(\hila)}

\newcommand{\Breg}{\widehat{B}_{\textup{reg}}}

\newcommand{\M}{\mathcal{M}}
\newcommand{\sigef}{\sigma(E,F)}
\newcommand{\sigfe}{\sigma(F,E)}
\newcommand{\hil}{\mathcal H}

\newcommand{\hila}{\hil^{}_A}

\newcommand{\hilab}{\hil_{A+B}}
\newcommand{\hilb}{\hil^{}_B}

\newcommand{\kil}{\mathcal K}

\DeclarePairedDelimiterX\sip[2]{(}{)}{#1\,\delimsize\vert\,#2}
\DeclarePairedDelimiterX\siptilde[2]{(}{)_{\!_{\widetilde{A}}}}{#1\,\delimsize\vert\,#2}
\DeclarePairedDelimiterX\sipf[2]{(}{)_{f}}{#1\,\delimsize\vert\,#2}
\DeclarePairedDelimiterX\sipg[2]{(}{)_{g}}{#1\,\delimsize\vert\,#2}
\DeclarePairedDelimiterX\siptw[2]{(}{)_{\tform+\wform}}{#1\,\delimsize\vert\,#2}
\DeclarePairedDelimiterX\set[2]{\{}{\}}{#1\,:\,#2}
\DeclarePairedDelimiterX\dual[2]{\langle}{\rangle}{#1,#2}
\DeclarePairedDelimiterX\sipa[2]{(}{)_{\!_A}}{#1\,\delimsize\vert\,#2}
\DeclarePairedDelimiterX\sipc[2]{(}{)_{\!_C}}{#1\,\delimsize\vert\,#2}
\DeclarePairedDelimiterX\sipab[2]{(}{)_{\!_{A+B}}}{#1\,\delimsize\vert\,#2}
\DeclarePairedDelimiterX\sipb[2]{(}{)_{\!_B}}{#1\,\delimsize\vert\,#2}
\newcommand{\anti}[1]{\bar{#1}'}

\newcommand{\limn}{\lim\limits_{n\rightarrow\infty}}

\allowdisplaybreaks
\title[Lebesgue decomposition of positive operators]{Operators on anti-dual pairs:\\ Lebesgue decomposition of positive operators}

\author[Zs. Tarcsay]{Zsigmond Tarcsay}
\thanks{The author was supported by the DAAD-Tempus PPP Grant ``Harmonic analysis and extremal problems'' and by the ``For the Young Talents of the Nation'' scholarship program (NTP-NFT\"O-17) of the Hungarian Ministry of Human
Capacities.}
\address{%
Zs. Tarcsay \\ Department of Applied Analysis  and Computational Mathematics\\ E\"otv\"os Lor\'and University\\ P\'azm\'any P\'eter s\'et\'any 1/c.\\ Budapest H-1117\\ Hungary}
\email{tarcsay@cs.elte.hu}



\begin{document}
\maketitle
\begin{abstract}
 In this paper we introduce and study absolute continuity and singularity of positive operators acting on anti-dual pairs. We  establish a general 
 theorem that  can be considered as a common generalization of various earlier Lebesgue-type decompositions.
 Different algebraic and topological characterizations of absolute continuity and singularity are supplied and also a complete description of  uniqueness of the decomposition is provided. We apply the developed decomposition theory to some concrete objects including  Hilbert space operators, Hermitian forms, representable functionals, and additive set functions.
\end{abstract}
\section{Introduction}
This paper is part of a unification project aiming to find a common framework and generalization for various results obtained in  different branches of functional analysis including extension, dilation and decomposition theory. One important class of such results are decomposition theorems analogous to the well known Lebesgue decomposition of measures. What do we mean about analogous?  In several cases, transformations of a given system can be grouped into two extreme classes according to the behavior with respect to their qualitative properties.
These particular classes are the so-called \emph{regular} transformations (i.e.,  transformations with \lq\lq nice" properties) and the so-called \emph{singular} ones (transformations that are hard to deal with). Of course, regularity and singularity may have multiple meanings  depending on the context.  A decomposition of an object into regular and singular parts is called a Lebesgue-type decomposition. 

In order to understand a structure better, it can be effective 
to characterize its regular and singular elements. This explains  why a regular-singular type decomposition theorem may have theoretic importance, especially when the corresponding regular part can be interpreted in a canonical way. The prototype of such results is the celebrated Radon-Nikodym  theorem  stating that every $\sigma$-finite measure  splits uniquely into absolutely continuous and singular parts with respect to any other measure, and the absolutely continuous part has an integral representation. Returning to the previous idea, the Radon-Nikodym theorem can be phrased as follows: if we want to decide whether a set function can be represented as a point function, we only need to know if it is   absolutely continuous or not. That is to say, in this concrete  situation, the appropriate regularity concept is absolute continuity.

In the last 50 years quite a number of authors have made significant contributions to the vast literature of non-commutative Lebesgue-Radon-Nikodym theory -- here we mention only Ando \cite{Ando}, Gudder \cite{Gudder}, Inoue \cite{Inoue}, Kosaki \cite{Kosaki} and Simon \cite{Simon}, and from the recent past Di Bella and Trapani \cite{trapani}, Corso \cites{rosario1,rosario2,rosario3}, ter Elst and Sauter \cite{tES},  Gheondea \cite{Gheondea1}, Hassi et al. \cites{Hassi2007,Hassi2009a,Hassi2009, Hassilebesgue2018},   Sebesty\'en and Titkos \cite{TitokRN}, Sz\H ucs \cite{Szucs2012}, Vogt \cite{vogt}. 

The purpose of the present paper is to develop and investigate an abstract decomposition theory that can be considered as a common generalization of many of the aforementioned results on Lebesgue-type decompositions. The key observation is that the corresponding absolute continuity and singularity concepts rely only on some topological and algebraic  properties of an operator acting between an appropriately chosen vector space and its conjugate dual. So that, the problem of decomposing Hilbert space operators, representable functionals, Hermitian forms and measures can  be transformed into the problem of decomposing such an abstract operator. 

 In this note we are going to investigate Lebesgue decompositions of positive operators on a so called anti-dual pair. Hence, for the readers sake, we gathered in Section 2 the most important  facts about anti-dual pairs and operators between them. We also provide here a variant of the famous Douglas factorization theorem. Section 3 contains the main result of the paper (Theorem \ref{T:main_Lebdecomp}), a direct generalization of Ando's Lebesgue decomposition theorem \cite{Ando}*{Theorem 1} to the anti-dual pair context. It states that every positive operator on a weak-* sequentially complete anti-dual pair splits into a sum of absolutely continuous and singular parts with respect to another positive operator. We also prove that,  when decomposing two positive operators with respect to each other,  the corresponding absolute continuous parts are always mutually absolutely continuous. 
In Section 4 we introduce the parallel sum of two positive operators and furnish a different approach to the Lebesgue decomposition in terms of the parallel addition. In Section 5 we establish two characterizations of absolute continuity:  Theorem \ref{T:almostdom=ac} is  of algebraic nature, as it relies on the order structure of positive operators. Theorem \ref{T:R-N} is rather topological in character: it states that a positive operator is absolutely continuous with respect to another if and only if it can be uniformly approximated with the other one in a certain sense. Section 6 is devoted to characterizations of singularity, Section 7 deals with the uniqueness of the decomposition.  To conclude the paper, in Section 8 we apply the developed decomposition theory to some concrete objects including  Hilbert space operators, Hermitian forms, representable functionals, and additive set functions.

\section{Preliminaries}

The aim of this chapter is to collect all the technical ingredients that are necessary to read the paper.
\subsection{Anti-dual pairs}
Let $E$ and $F$ be complex vector spaces which are intertwined via a sesquilinear function $$\dual\cdot\cdot:F\times E\to\dupC,$$ which separates the points of $E$ and $F$. We shall refer to $\dual\cdot\cdot$ as \emph{anti-duality} and the triple  $(E,F,\dual\cdot\cdot)$ will be called an \emph{anti-dual pair} and shortly denoted by $\dual FE$. In this manner we may speak about \emph{symmetric} and, first and foremost, \emph{positive operators} from $E$ to $F$. Namely, we call an operator $A:E\to F$ \emph{symmetric}, if 
\begin{equation*}
    \dual{Ax}{y}=\overline{\dual{Ay}{x}},\qquad x,y\in E,
\end{equation*}
furthermore, $A$ is said to be \emph{positive}, if its ``quadratic form'' is positive semidefinite: 
\begin{equation*}
    \dual{Ax}{x}\geq 0,\qquad x\in E.
\end{equation*}
Clearly, every positive operator is symmetric. 

Most natural anti-dual pairs arise in the following way. Let $ \bar E^*$ denote the conjugate algebraic dual of a complex vector space $E$ and let $F$ be a separating subspace of $\bar E^*$. Then 
\begin{equation*}
    \dual{f}{x}\coloneqq f(x),\qquad x\in E, f\in F
\end{equation*}
defines an anti-duality, and the pair $\dual FE$ so obtained is called the natural anti-dual pair. (In fact, every anti-dual pair can be regarded as a natural anti-dual pair when $F$ is identified with $\widehat F$, the set consisting of the conjugate linear functionals $\dual f\cdot$, $f\in F$.)
Our prototype of anti-dual pairs is the system $((\hil,\hil),\sip{\cdot}{\cdot})$ where $\hil$ is a Hilbert space with inner product $\sip{\cdot}{\cdot}$.

Just as in the dual pair case (see e.g. \cite{Schaefer}), we may endow $E$ and $F$ with the corresponding weak topologies $\sigef$, resp. $\sigfe$, induced by the families $\set{\dual f\cdot }{f\in F}$, resp. $\set{\dual\cdot x}{x\in E}$. Both $\sigef$ and $\sigfe$ are locally convex Hausdorff topologies such that 
\begin{equation}\label{E:E'=F}
    \anti E=F,\qquad F'=E,
\end{equation}
where $F'$ and $\anti E$ refer to the topological dual and anti-dual space of $F$ and $E$, respectively, and  the vectors  $f\in F$ and $x\in E$ are identified  with $\dual f\cdot$, and  $\dual \cdot x$, respectively.  We also recall  the useful property of weak topologies that, for a topological vector space $(V,\tau)$, a linear operator $T:V\to F$ is $\sigfe$-contionuous if and only if 
\begin{equation*}
    T_x(v)\coloneqq \dual{Tv}{x},\qquad v\in V
\end{equation*}
is continuous for every $x\in E$. 

This fact and \eqref{E:E'=F} enables us to define the adjoint (that is, the topological transpose) of a weakly continuous operator. Let $\dual{F_1}{E_1}$ and $\dual{F_2}{E_2}$ be anti-dual pairs and  $T:E_1\to F_2$ a weakly continuous linear operator, then the (necessarily weakly continuous) linear operator $T^*:E_2\to F_1$ satisfying \begin{equation*}
    \dual{Tx_1}{x_2}_2=\overline{\dual{T^*x_2}{x_1}_1},\qquad  x_1\in E_1,x_2\in E_2
\end{equation*} 
is called the adjoint of $T$. In particular, the adjoint of a weakly continuous operator $T:E\to F$ emerges  again as an operator of this type. The set of weakly continuous linear operators $T:E\to F$ will be denoted by $\lef$. An operator $T\in \lef$ is called self-adjoint if $T^*=T$. It is immediate that every symmetric operator (hence every positive operator) is weakly continuous and self-adjoint. 

Finally, we recall that a topological vector space $(V,\tau)$ is called complete if every Cauchy net in $V$ is convergent. Similarly, $V$ is sequentially complete if every Cauchy sequence in $V$ is convergent. We shall call the anti-dual pair $\dual FE$ weak-* (sequentially) complete if $(F,\sigfe)$ is (sequentially) complete. It is easy to see that $\dual{\bar E^*}{E}$ is always weak-* complete. It can be deduced from the Banach-Steinhaus theorem that, for a Banach space $E$,  $\dual{\anti E}{E}$ is weak-* sequentially complete (but not weak-*  complete unless $E$ is finite dimensional). 
\subsection{Factorization of positive operators.}\label{Sub:factorizepositive}
Let $\dual FE$ be an anti-dual pair and $A:E\to F$ a positive operator. As we have already mentioned, $A\in \lef$ and $A=A^*$. Now we give the prototype of positive operators.
Let $\hil$ be a complex Hilbert space and let $T:E\to\hil$  be a weakly continuous (i.e., $\sigma(E,F)$-$\sigma(\hil,\hil)$ continuous) linear operator, then the adjoint operator $T^*:\hil\to F$ is again weakly continuous and the product $T^*T\in\lef$ is positive:
    \begin{equation*}\label{E:T*T}
        \dual{T^*Tx}{x}=\sip{Tx}{Tx}\geq0,\qquad x\in E.
    \end{equation*}
On a weak-* sequentially complete anti-dual pair $\dual FE$, every positive operator $A\in\lef$ can be written as $A=T^*T$. We sketch here the proof of this fact because we will use the construction continuously; for more details the reader is referred to \cite{TZS-TT:KreinNeumannADP}*{Theorem 3.1}.

Let $\dual FE$ be a  weak-* sequentially complete anti-dual pair and let $A\in\lef$ be a positive operator. Endow the range space $\ran A$ with the following inner product:
\begin{equation*}
    \sipa{Ax}{Ay}\coloneqq \dual{Ax}{y},\qquad x,y\in E.
\end{equation*}
One can show that $\sipa\cdot\cdot$ is well defined and positive definite, hence $(\ran A,\sipa\cdot\cdot)$ is a pre-Hilbert space. Let $\hila$ denote its Hilbert completion so that $\ran A\subseteq \hila$ forms a norm dense linear subspace. The canonical embedding operator
\begin{equation}\label{E:J_A}
        J^{}_A(Ax)=Ax,\qquad x\in E,
    \end{equation}
of $\ran A\subseteq \hila$ into $F$ is weakly continuous, hence $J_A$ extends to an everywhere defined weakly continuous operator because of weak-* sequentially completeness of $F$. We continue to write $J_A\in\mathscr L(\hila,F)$ for this extension. The adjoint operator $J_A^*\in\mathscr{L}(E,\hila)$ admits the canonical property 
\begin{equation}\label{E:J*}
        J_A^*x=Ax\in\hila,\qquad x\in E,
\end{equation}
that leads  to the useful factorization of $A$:
\begin{equation}\label{E:JAJA}
    A=J_A^{}J_A^*.
\end{equation}

\subsection{Range of the adjoint operator} 
Operators of type $T\in\mathscr{L}(E,\hil)$ will play a  peculiar role in the theory of positive operators, as we have seen, every positive operator $A$ on a weak-* sequentially complete anti-dual pair admits a factorization $A=T^*T$ through a Hilbert space $\hil$. In this section we describe the range of  the adjoint operator $T^*\in\mathscr L(\hil,F)$. The key result is a variant to Douglas' famous range inclusion theorem \cite{Douglas} (for further  generalizations to Banach space setting see Barnes \cite{Barnes} and Embry \cite{Embry}).

\begin{theorem}\label{T:Douglas}
Let $\dual FE$ be an anti-dual pair and let $\hil_1,\hil_2$ be Hilbert spaces. Given two weakly continuous operators $T_j\in \mathscr L(\hil_j,F)$ ($j=1,2$)  the following assertions are equivalent:
\begin{enumerate}[label=\textup{(\roman*)}, labelindent=\parindent]
\item $\ran T_1\subseteq \ran T_2$,
\item there is a constant $\alpha\geq0$ such that $$\|T_1^*x\|^2\leq \alpha \|T_2^*x\|^2,  \qquad x\in E,$$
\item for every $h_1\in \hil_1$ there is a constant $\alpha_{h_1}\geq0$ such that $$\abs{\dual{T_1h_1}{x}}^2\leq \alpha_{h_1}\|T_2^*x\|^2,\qquad x\in E,$$ 
\item there is a bounded operator $D:\hil_1\to\hil_2$ such that $$T_1=T_2D.$$
\end{enumerate}
Moreover, if any (hence all) of (i)-(iv) is valid, then there is a unique $D$ such that 
\begin{enumerate}[label=\textup{(\alph*)}]
\item $\ran D\subseteq (\ker T_2)^{\perp}$,
\item $\ker T_1=\ker D$,
\item $\|D\|^2= \inf\set{\alpha\geq0}{\|T_1^*x\|^2\leq \alpha \|T_2^*x\|^2, (x\in E)}$.
\end{enumerate} 
 \begin{figure}[H]
\begin{tikzpicture}
\draw (10,2) node (A) {$\hil_1$};
\draw (10,0) node (B) {$\hil_2$};
\draw[thick,-latex] (A) -- (B)
 node[pos=0.5,left] {$D$};
\draw (12,1) node (C) {$F$};
\draw[thick,-latex] (A) -- (C)
 node[pos=0.5,above] {$T_1$\,};
\draw[thick,-latex] (B) -- (C)
 node[pos=0.5,below] {\,$T_2$};
\end{tikzpicture}
\caption{Factorization
of $T_1$ along $T_2$}
\end{figure}
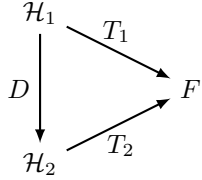
\end{theorem}
\begin{proof}
Implications (i)$\Rightarrow$(iii), (iv)$\Rightarrow$(ii)$\Rightarrow$(iii) and (iv)$\Rightarrow$(i) are immediate. We only prove (iii)$\Rightarrow$(iv): fix $h_1$ in $\hil_1$ and define a conjugate linear functional $\phi:\ran T_2^*\to\dupC$ by
\begin{align*}
\phi(T_2^*x)\coloneqq \dual{T_1h_1}{x},\qquad x\in E.
\end{align*}
By  (iii) one concludes that $\phi$ is well defined and continuous by norm bound $\|\phi\|\leq \sqrt{\alpha_{h_1}}$. The Riesz representation theorem yields then a unique representing vector $Dh_1\in \overline{\ran T_2^*}$ such that 
\begin{align*}
\dual{T_1h_1}{x}=\sip{Dh_1}{T_2^*x},\qquad x\in E, h_1\in \hil_1.
\end{align*}
It is easy to check that $D:\hil_1\to\hil_2$ is linear and that $T_2D=T_1$. Our only duty is to prove that $D$ is continuous. Take $h_2\in(\ran T_2^*)^{\perp}$ and $x\in E$, then   for any $h_1\in\hil_1$ 
\begin{align*}
\sip{Dh_1}{h_2+T_2^*x}=\sip{Dh_1}{T_2^*x}=\sip{h_1}{T_1^*x}.
\end{align*}
This means that the domain of $D^*$ includes the dense set $\ran T_2^*+(\ran T_2^*)^{\perp}$, hence $D$ is closable. By the closed graph theorem we conclude that $D$ is continuous. 

Observe also that $D$ obtained above fulfills conditions (a)-(c) above. Indeed, (a) and (b) are straightforward, and if (ii) holds for some $\alpha\geq0$ then for $x\in E$ and $u\in (\ran T_2^*)^{\perp}$ we have $D^*u=0$ by (a). Consequently,
\begin{align*}
\|D^*(T_2^*x+u)\|^2=\|T_1^*x\|^2\leq \alpha\|T_2^*x\|^2\leq \alpha\|T_2^*x+u\|^2,
\end{align*}
which implies $\|D^*\|^2\leq  \alpha$, hence $D$ satisfies (c). Finally, the uniqueness follows easily from (b).
\end{proof}
\begin{corollary}\label{C:factor}
 Let $\dual FE$ be a weak-* sequentially complete anti-dual pair. If $A,B\in\lef$ are positive operators such that $B\leq A$, then there is a unique positive contraction $C\in\bha$ such that $B=J_ACJ_A^*$. 
\end{corollary}
\begin{proof}
Let $\hilb$ stand for the auxiliary Hilbert space obtained by the  procedure of Subsection \ref{Sub:factorizepositive}, with $A$ replaced by $B$.  Since $B\leq A$, we have  $\|J_B^*x\|^2\leq \|J_A^*x\|^2$ for every $x\in E$. By Theorem \ref{T:Douglas} there exists a bounded operator $D\in\mathscr B(\hila,\hilb)$, $\|D\|\leq 1$, such that $J_B=J_AD$, hence $C\coloneqq DD^*$ satisfies  
\begin{eqnarray*}
 J_ACJ_A^*=J_BJ_B^*=B. 
\end{eqnarray*}
The uniqueness of $C$ follows easily from the fact that $\ran J_A^*=\ran A$ is dense in $\hila$.
\end{proof}

The following range description of the adjoint operator  is similar in
spirit to \cite{Smulian}, cf. also  \cite{Sebestyen83}.

\begin{lemma}\label{L:adjointrange}
Let $\dual FE$ be an anti-dual pair, $\hil$ a Hilbert space and $T:E\to \hil$ a weakly  continuous linear operator. A vector $y\in F$ belongs to the range of $T^*$  if and only if there exists $m_y\geq0$ such that 
 \begin{equation}\label{E:adjoitrange}
     \abs{\dual{y}{x}}^2\leq m_y \|Tx\|^2,\qquad x\in E.
\end{equation}
\end{lemma}
\begin{proof}
Assume first that $y=T^*h$ for some $h\in\hil$, then
\begin{eqnarray*}
\abs{\dual{y}{x}}^2=\abs{\dual{T^*h}{x}}^2=\abs{\sip{Tx}{h}}^2\leq \|h\|^2\|Tx\|^2,\qquad x\in E,
\end{eqnarray*}
hence (i) implies (ii). Conversely, (ii) expresses precisely that the correspondence
\begin{eqnarray*}
Tx\mapsto \dual{y}{x},\qquad x\in E,
\end{eqnarray*}
defines a continuous anti-linear functional from $\ran T\subseteq \hil$ to $\dupC$. The Riesz representation theorem yields  a vector $h\in\overline{\ran T}$ such that 
\begin{eqnarray*}
 \dual{y}{x}=\sip{h}{Tx}=\dual{T^*h}{x}, \qquad x\in E.
\end{eqnarray*}
Consequently, $y=T^*h\in\ran T^*$.
\end{proof}
\subsection{Linear relations in Hilbert spaces}\label{Sub:linrel}
If $A,B$ are positive operators on the anti-dual pair $\dual FE$ then we can associate  the auxiliary Hilbert spaces $\hila,\hilb$ with them along the procedure given in Subsection \ref{Sub:factorizepositive}. 
The vast majority of the results in this paper relies on some topological properties of a ``mapping'' sending $Ax$ of $\hila$ into $Bx$ of $\hilb$. In general, this map is not a function (and thus not a bounded operator). Such ``multivalued'' operators, i.e, linear subspaces  of a product Hilbert space $\hil\times\kil$ are called linear relations.  In this subsection we gather some basic notions and results of the theory of linear relations. For a comprehensive
treatment  on linear relations one may refer to  \cite{Arens} and \cite{Hassi2007}. 

A linear relation  between two Hilbert spaces $\hil$ and $\kil$ is a linear subspace $T$ of the Cartesian product $\hil\times\kil$. Accordingly, $T$ is called a closed linear relation if it is a closed linear subspace of $\hil\times\kil$. If $T$ is a linear operator from $\hil$ to $\kil$ then the graph of $T$ is a linear relation.  Conversely, a linear relation $T$ is (the graph of) an operator if and only if $(0,k)\in T$ implies $k=0$ for every $k\in\kil$. In other words, a linear relation $T$ is an operator if its multivalued part
\begin{equation*}
    \mul T\coloneqq \set{k\in\kil}{(0,k)\in T}
\end{equation*}
is trivial. The domain, kernel and range of a linear relation $T$, denoted by $\dom T$, $\ker T$ and $\ran T$, respectively, are defied in the obvious manner. A relation $T$ is called closable if its closure $\overline{T}$ is an operator, or equivalently, if
\begin{equation*}
\mul \overline{T}=\set{k\in K}{\exists (h_n,k_n)_{\nen} \subset  T, h_n\to0, k_n\to k}
\end{equation*}
is  trivial. In the sequel, we shall also need the concept of the adjoint of a linear relation. To this aim let us introduce the unitary operator 
\begin{equation*}
V(h,k)\coloneqq (-k,h),\qquad h\in\hil, k\in\kil
\end{equation*}
from $\hil\times \kil$ to $\kil\times\hil$. The adjoint of a linear relation $T\subseteq \hil\times\kil$ is  given by 
\begin{equation*}
    T^*\coloneqq [V (T)]^{\perp}=\set{(-k,h)}{(h,k)\in T}^{\perp},
\end{equation*}
that agrees with the original concept of the adjoint transformation introduced by J. von Neumann if $T$ is a densely defined operator. 
Observe immediately that $T^*$ is always  closed such that $T^{**}=\overline{T}.$  
For a pair of vectors $(g,f)\in\kil\times\hil$, relation $(g,f)\in T^*$ means that
\begin{equation}\label{E:adjoint_relation}
\sip{k}{g}=\sip{h}{f},\qquad \textrm{for all $(h,k)\in T$.}
\end{equation}
In a full analogy with the operator case, the domain of the adjoint relation consists of those vectors $g$ such that 
\begin{equation}\label{E:domT*}
\abs{\sip{k}{g}}^2\leq m_g \sip{h}{h}\qquad \textrm{for all $(h,k)\in T$}
\end{equation}
holds for some $m_g\geq0$. Furthermore, we have the following useful relations:
\begin{equation}\label{E:muldom}
\mul T^*=(\dom T)^{\perp}\qquad \textrm{\and}\qquad \overline{\dom T^*} =(\mul T^{**})^{\perp}.
\end{equation}
In particular, the adjoint of a densely defined relation is a closed operator and the adjoint of a closable operator is densely defined. 
Let $P$ denote the orthogonal projection of $\kil$ onto $\mul{\overline{T}}$. The regular part $T_\textrm{reg}$ of $T$ is defined to be the linear relation
\begin{eqnarray}\label{E:Treg}
T_\textrm{reg}\coloneqq \set{(h,(I-P)k)}{(h,k)\in T}.
\end{eqnarray} 
Actually, it can be proved that $T_\textrm{reg}$ is a closable operator   and its closure satisfies  
\begin{eqnarray}\label{E:Treg-lezart}
\overline{T_\textrm{reg}}=(\overline{T})_{\textrm{reg}},
\end{eqnarray}
see \cite{Hassi2007}*{Theorem 4.1 and  Proposition 4.5}. In particular, the regular part of a closed linear relation is itself closed. 
\section{Lebesgue decomposition theorem for positive operators}\label{Sec:3}

Modeled by the Lebesgue--Radon--Nikodym theory of positive operators on a Hilbert space (see e.g. \cite{Ando} or \cite{Tarcsay_Leb}) we can introduce the concepts of absolute continuity and singularity of positive operators on an anti-dual pair. Let $A$ and $B$ be positive operators on an anti-dual pair $\dual FE$. We say that $B$ is absolutely continuous with respect to $A$ (in notation, $A\ll B$) if for any sequence $\seq{x}$ of $E$,
 \begin{equation*}
    \dual{Ax_n}{x_n}\to0\qquad \textrm{and}\qquad \dual{B(x_n-x_m)}{x_n-x_m}\to0\qquad (n,m\to\infty)
 \end{equation*}
 imply $\dual{Bx_n}{x_n}\to0$. On the other hand, we say that $A$ and $B$ are mutually singular (in notation, $A\perp B$) if $C\leq A$ and $C\leq B$ imply $C=0$ for any positive operator $C\in\lef$.
 
The main purpose of this section is to establish an extension of Ando's Lebesgue decomposition theorem \cite{Ando}*{Theorem 1}. This states that every positive operator $B$ on a weak-* sequentially complete anti-dual pair admits a decomposition $B=B_a+B_s$ where $B_a\ll A$ and $B_s\perp A$. Before passing to the proof, let us make a few remarks.

The following constraction is analogous to the one developed in \cite{Tarcsay_Leb}. Let us consider  the Hilbert spaces $\hila, \hilb$ and the linear operators $J_A,J_B$, associated with $A$ and $B$, respectively. Introduce the closed linear relation
 \begin{equation}\label{E:Bkalap}
    \widehat{B}\coloneqq \overline{\set{(Ax,Bx)\in\hila\times\hilb}{x\in E}}
 \end{equation}
 from $\hila$ to $\hilb$, and denote its multivalued part  by $\M$:
\begin{equation*}
    \M\coloneqq \set{\xi\in\hilb}{(0,\xi)\in \widehat{B}}.
 \end{equation*}
 According to what has been said in Subsection \ref{Sub:linrel}, $\M$ is a closed linear subspace of $\hilb$ and
  one easily verifies that
 \begin{equation}\label{E:M}
    \M=\set{\xi\in\hilb}{\exists \seq{x}~\textrm{of $E$}, \dual{Ax_n}{x_n}\to0, Bx_n\to\xi~\textrm{in~$\hilb$}}.
 \end{equation}
 It is easy to check that $B\ll A$ if and only if $\widehat B$ is a closed operator, or equivalently, if $\M=\{0\}$. Furthermore, since $\ran A\subseteq \dom \widehat B$, the adjoint relation $\widehat{B}^*$ is always a single-valued operator from $\hilb$ to $\hila$ such that 
 \begin{equation}\label{E:domBperp=M}
     (\dom \widehat B^*)^\perp =\M.
 \end{equation}
 
The next lemma describes the domain of $\widehat B^*$:
 \begin{lemma}\label{L:domBkalap}
For a vector $\xi\in\hilb$ the following assertions are equivalent:
\begin{enumerate}[label=\textup{(\roman*)}, labelindent=\parindent]
 \item $\xi\in\dom \widehat{B}^*$,
 \item there exists $m_{\xi}\geq 0$ such that $\abs{\sipb{Bx}{\xi}}^2\leq m_{\xi}\dual{Ax}{x}$ for all $x\in E$,
 \item $J_B\xi\in\ran J_A$.
\end{enumerate}
In any case, 
\begin{equation}\label{E:JABkalap}
J_A\widehat{B}^*\xi=J_B\xi,\qquad \xi\in\dom \widehat{B}^*.
\end{equation}
\end{lemma}
\begin{proof}
The equivalence between (i) and (ii) is clear due to  \eqref{E:domT*} and   the equivalence between (ii) and (iii) follows from  Lemma \ref{L:adjointrange}. Finally, for $\xi\in\dom \widehat{B}^*$  and $x\in E$ we have 
\begin{align*}
\dual{J_A\widehat{B}^*\xi}{x}&=\sipa{\widehat{B}^*\xi}{Ax}=\sipb{\xi}{Bx}=\dual{J_B\xi}{x},
\end{align*}
that proves \eqref{E:JABkalap}.
\end{proof}
Let   $P$ stand for the orthogonal projection of $\hilb$ onto $\M$ and set
\begin{equation}\label{E:Bhat}
    \widehat{B}_{\textup{reg}}\coloneqq \set{(\zeta,(I-P)\xi)}{(\zeta,\xi)\in \widehat{B}}.
 \end{equation}
 Since $\widehat{B}_{\textup{reg}}$ is the regular part \eqref{E:Treg} of $\widehat{B}$,  \cite{Hassi2007}*{Theorem 1} and identity \eqref{E:Treg-lezart} yield the following result:
 \begin{proposition}\label{P:closable}
  $\Breg$ is a densely defined closed linear operator between $\hila$ and $\hilb$ such that 
  \begin{equation*}
     \widehat{B}_{\textup{reg}}=\overline{\set{(Ax,(I-P)(Bx))}{x\in E}}.
 \end{equation*}
 \end{proposition}\label{}
 We can now prove our main result that establishes a Lebesgue-type decomposition theorem for positive operators on a weak-* sequentially complete anti-dual pair:
 \begin{theorem}\label{T:main_Lebdecomp}
    Let $A,B$ be positive operators on a weak-* sequentially complete  anti-dual pair $\dual FE$. Let $P$ stand for the the orthogonal projection of $\hilb$ onto $\M$, then
     \begin{equation}\label{E:Ba_Bs}
        B_a\coloneqq J_B^{}(I-P)J_B^*\qquad \textrm{and} \qquad B_s\coloneqq J_BPJ_B^*
     \end{equation}
     are positive operators such that $B=B_a+B_s$, $B_a$ is $A$-absolutely continuous and $B_s$ is $A$-singular.  Furthermore, $B_a$ is the greatest element of the set of those positive operators $C\in\lef$ such that $C\leq B$ and $C\ll A$.
      \begin{figure}[H]
\begin{tikzpicture}
\draw (10,0) node (A) {$E$};
\draw (13,0) node (B) {$F$};
\draw (10,2) node (C) {$\hilb$};
\draw (13,2) node (D) {$\M^\perp$};
\draw[thick,-latex] (A) -- (B)
 node[pos=0.5,below] {$B_a$};
\draw[thick,-latex] (C) -- (D)
 node[pos=0.5,above] {$I-P$};
\draw[thick,-latex] (A) -- (C)
 node[pos=0.5,left] {$J_B^*$};
\draw[thick,-latex] (D) -- (B)
 node[pos=0.5,right] {$J_B$};
\end{tikzpicture}
\caption{Factorization of the absolute continuous part}
\end{figure}
 \end{theorem}
 \begin{proof}
    It is clear that  $B_a,B_s\in\lef$ are positive operators such that $B=B_a+B_s$. In order to prove  absolute continuity of $B_a$, we observe  that
    \begin{align*}
        \dual{B_ax}{x}&=\sipb{(I-P)J_B^*x}{J_B^*x}=\sipb{(I-P)(Bx)}{(I-P)(Bx)}\\
                      &=\sipb{\Breg(Ax)}{\Breg(Ax)},
    \end{align*}
    for $x\in E$, hence $B_a$ is $A$-absolutely continuous according to Proposition \ref{P:closable}.

    Our next claim is to show the maximality of $B_a$. Let us consider  a positive operator $C\in\lef$ such that $C\leq B$ and $C\ll A$. By Corollary \ref{C:factor}, there is a unique positive operator $\widetilde{C}\in\mathscr{B}(\hilb)$, $\|\widetilde{C}\|\leq1$, such that $C=J_B\widetilde{C}J_B^*$. In particular we have 
    \begin{equation*}
        \dual{Cx}{y}=\sipb{\widetilde{C}^{1/2}(Bx)}{\widetilde{C}^{1/2}(By)},\qquad x,y\in E.
    \end{equation*}
    We claim that
    \begin{equation}\label{E:kerCkalap}
        \M\subseteq \ker\widetilde{C}.
    \end{equation}
    For let $\xi\in\M$ and consider a sequence $\seq{x}$ of $E$ such that
    \begin{equation*}
        \sipa{Ax_n}{Ax_n}\to 0\qquad \textrm{and}\qquad Bx_n\to \xi\in\hilb.
    \end{equation*}
    By continuity, $\widetilde{C}^{1/2}(Bx_n)\to\widetilde{C}^{1/2}\xi$, and by  $A$-absolute continuity,
    \begin{equation*}
    \sipb{\widetilde{C}^{1/2}(Bx_n)}{\widetilde{C}^{1/2}(Bx_n)}=\dual{Cx_n}{x_n}\to0,
    \end{equation*}
    whence $\widehat{C}\xi=0$. This proves \eqref{E:kerCkalap}. Let now $x\in E$. By \eqref{E:kerCkalap},  $\widetilde{C}^{1/2}P=0$. Consequently,
    \begin{align*}
        \dual{Cx}{x}&=\sipb{\widetilde{C}^{1/2}(Bx)}{\widetilde{C}^{1/2}(Bx)}=\sipb{\widetilde{C}^{1/2}(I-P)(Bx)}{\widetilde{C}^{1/2}(I-P)(Bx)}\\
        &\leq \sipb{(I-P)(Bx)}{(I-P)(Bx)}=\dual{J_B^{}(I-P)J_B^*x}{x}\\
        &=\dual{B_ax}{x},
    \end{align*}
    whence $C\leq B_a$, as it is stated.

    Finally, in order to prove that that $B_s$ and $A$ are mutually singular, let  $C\in\lef$ be a positive operator such that $C\leq A$ and $C\leq B_s$. Then $C+B_a\leq B$ so that $C+B_a$ is $A$-absolutely continuous. By the maximality of $B_a$ we conclude that $C+B_a\leq B_a$, i.e., $C=0$.
 \end{proof}
\begin{remark}
Observe that
\begin{align*}
\Breg J_A^*x=(I-P)Bx=(I-P)J_B^*x
\end{align*}
 for any $x$ in $E$, whence we obtain yet another  useful factorization of the absolute continuous part:
\begin{equation}\label{E:BJABJA}
 B_a=(\widehat{B}_{\textup{reg}}J_A^*)^*(\widehat{B}_\textup{reg}J_A^*).
\end{equation} 
\end{remark}
We close the section with an interesting property of the absolute continuous part. Suppose that $A,B$ are positive operators and let $B=B_a+B_s$ be the Lebesgue decomposition of $B$ with respect to $A$ in virtue of to Theorem \ref{T:main_Lebdecomp}, i.e., 
\begin{equation*}
    B_a=J_B^{}(I-P)J_B^*
\end{equation*} 
and $B_s=J_BPJ_B^*$. Here we have  $B_a\ll A$. Interchanging the roles of $A$ and $B$, by the same process we may take the Lebesgue decomposition of $A$ with respect to $B$, namely, $A=A_a+A_s$. We shall prove that the absolutely continuous parts $A_a$ and  $B_a$ are absolutely continuous with respect to each other, i.e., $B_a\ll A_a$ and $A_a\ll B_a$. This surprising property was discovered by T. Titkos  in context of nonnegative forms \cite{Titkos_content} and measures \cite{Titkos_howto}. Theorem \ref{T:mutualabs} below is not only a generalization of this fact, it also reproves these results with a completely different technique.
\begin{lemma}
Let $\hil$ and $\kil$ be Hilbert spaces and let $T$ be a closed linear relation between them and denote by $P_T$ and $Q_T$ the orthogonal projections onto $\mul T$ and $\ker T$, respectively. Then 
\begin{equation*}
    S\coloneqq \set{((I-Q_T)\xi,(I-P_T)\eta)}{(\xi,\eta)\in T}
\end{equation*}
is (the graph of) a one-to-one closed operator.
\end{lemma}
\begin{proof} 
It is easy to see that  $\ker T=\mul T^{-1}$ and that $\ker T=\ker T_{\textup{reg}}$, furthermore the regular part of a closed linear relation is closed itself, hence 
\begin{equation*}
    S=(((T_{\textup{reg}})^{-1})_{\textup{reg}})^{-1}
\end{equation*}
is a one-to-one closed operator.
\end{proof}
\begin{theorem}\label{T:mutualabs}
Let $\dual FE$ be a weak-* sequentially anti-dual pair and  let $A,B\in\lef$ be positive operators. Then we have 
\begin{equation*}
    A_a\ll B_a \qquad\mbox{and}\qquad B_a\ll A_a.
\end{equation*}
\end{theorem}
\begin{proof}
Let us continue to write $P$ for the orthogonal projection onto $\mul \widehat{B}$ and let $Q$ be the orthogonal projection onto $\ker \widehat{B}$. Then, according to the preceding lemma, 
\begin{equation*}
    S\coloneqq \set{((I-Q)\xi,(I-P)\eta)}{(\xi,\eta)\in\widehat{B}}
\end{equation*}
is (the graph of) a one-to-one closed linear operator from $\hila$ to $\hilb$. Since we have $\ker\widehat{B}=\mul \widehat{B}^{-1}$, it follows that 
\begin{equation*}
    B_a=J^{}_B(I-P)J_B^*,\qquad A_a=J_A^{}(I-Q)J_A^*.
\end{equation*}
Consider a sequence $\seq x$ in $E$ such that $\dual{A_ax_n}{x_n}\to0$ and $\dual{B_a(x_n-x_m)}{x_n-x_m}\to0$, then $(I-Q)Ax_n\to0$ in $\hila$ and $(I-Q)Bx_n\to\eta$ for some $\eta\in\hilb$. Since $S$ is closable it follows that $\eta=0$ and hence that $\dual{B_ax_n}{x_n}\to0$, thus $B_a\ll A_a$. A very similar reasoning shows that  $A_a\ll B_a$, but this time the closability of $S^{-1}$ is used. 
\end{proof}
\begin{remark}
We have only proved that the canonical absolute continuous parts $A_a$ and $B_a$ have the property of being mutually absolute continuous. As we shall see, the Lebesgue decomposition is not unique in general, so there might exist other Lebesgue-type decompotitions differing from what we  have constructed in Theorem \ref{T:main_Lebdecomp}. The statement of Theorem \ref{T:mutualabs} is certainly  not true for the absolutely continuous parts of such  Lebesgue decompositions.
\end{remark}
 \section{The parallel sum }
 Ando's key notion  in   establishing his Lebesgue-type decomposition theorem  was the so called parallel sum of two positive operators. Inspired by his treatment, Hassi, Sebesty\'en, and de Snoo \cite{Hassi2009} proved an analogous result for nonnegative Hermitian forms by means of the parallel sum as well. Parallel addition may also be defined in various areas of functional analysis, e.g. for   measures, representable positive functionals on a $^*$-algebra, and for positive operators from a Banach space to its topological anti-dual,  see \cites{Szucs_abs, Tarcsay_parallel, Titkos_means}. In what follows we  provide a common generalization of those concepts.  
 
 The parallel sum $A:B$ of two bounded positive operators on a Hilbert space can be introduced in various ways, see eg. \cites{AndersonDuffin, fillmore, Pekarev, Kosaki2018}, cf. also \cites{Djikic, Antezena}. 
 Its quadratic form can be obtained via the formula 
  \begin{equation}\label{E:A:B-Hilbert}
    \sip{(A:B)x}{x}=\inf\set{\sip{A(x-y)}{x-y}+\sip{By}{y}}{y\in\hil}, 
 \end{equation}
 that uniquely determines the operator $A:B$. Therefore, it seems  natural  to introduce  the parallel sum of two positive operators in the anti-dual pair context as an operator whose  quadratic form is \eqref{E:A:B-Hilbert} (the inner product replaced by anti-duality, of course). 
 
 The existence of such an operator is established in the following result:
\begin{theorem}
    Let $\dual FE$ be a weak-* sequentially complete anti-dual and let $A,B\in\lef$ be  positive operators.  There exists a unique positive operator $A:B\in\lef$, called the parallel sum of $A$ and $B$, such that 
    \begin{equation}\label{E:parallel}
    \dual{(A:B)x}{x}=\inf\set{\sip{A(x-y)}{x-y}+\sip{By}{y}}{y\in E},\qquad x\in E.
 \end{equation}
\end{theorem}
\begin{proof}
    Let us consider the product Hilbert space $\hila\times\hilb$ and the weakly continuous operator $V_A:\hila\times\hilb\to F$ arising from the densely defined one 
    \begin{equation*}
         V_A(Ax,By)\coloneqq Ax, \qquad x,y\in E.
    \end{equation*}
    A straightforward calculation shows that the  adjoint $V_A^*\in\mathscr{L}(E;\hila\times\hilb)$ fulfills
   \begin{equation}\label{E:V_A*}
              V_A^*x=(Ax,0)\in\hila\times\hilb,\qquad x\in E.
  \end{equation}
   Consider the orthogonal projection $Q$ of $\hila\times\hilb$ onto $\widehat{B}^{\perp}$. The positive operator $V_A^{}QV_A^*\in \lef$ satisfies then
   \begin{align*}
    \dual{V_A^{}QV_A^*x}{x}&=\|Q(Ax,0)\|^2_{\hila\times\hilb}=\dist^2((Ax,0),\widehat{B})\\
                        &=\inf\set[\big]{\|(Ax,0)-(Ay,By)\|_{\hila\times\hilb}^2}{y\in E}\\
                          &=\inf\set{\sipa{A(x-y)}{A(x-y)}+\sipb{By}{By}}{y\in E}\\
                        &=\inf\set{\dual{A(x-y)}{x-y}+\dual{By}{y}}{y\in E}.
   \end{align*}
    Hence $A:B\coloneqq V_A^{}QV_A^*$ fulfills \eqref{E:parallel}.
\end{proof}
In the next proposition we collected some basic properties of parallel addition:
\begin{proposition}\label{P:propertiesA:B}
Let $\dual FE$ be a weak-* sequentially complete anti-dual pair and let $A,B\in\lef$ be positive operators. Then
\begin{enumerate}[label=\textup{(\alph*)}]
    \item $A:B=B:A$,
    \item $A:B\leq A$ and $A:B\leq B$,
    \item $A_1\leq A_2$ and $B_1\leq B_2$ imply $A_1:B_1\leq A_2: B_2$.
\end{enumerate}
\end{proposition}
\begin{proof}
Replacing $y$ by $x-y$ in \eqref{E:parallel} yields
\begin{equation*}
\dual{(A:B)x}{x}=\dual{(B:A)x}{x}\qquad  \textrm{for all $x\in E$},
\end{equation*}
that gives just (a). Properties (b) and (c) are immediate from \eqref{E:parallel}.
\end{proof}
Note that the definition $A:B=V_A^{}QV_A^*$ of the parallel sum shows some asymmetry in $A$ and $B$, although we have $A:B=B:A$. If we define $V_B:\hila\times\hilb\to F$ by 
\begin{equation*}
    V(Ax,By)=By,\qquad x,y\in E,
\end{equation*}
then we will get $V_BQV_B^*=B:A$ and hence 
\begin{equation*}
    V_B^{}QV_B^*=V_A^{}QV_A^*=A:B.
\end{equation*}
\begin{lemma}\label{L:monoton}
    Let $\dual FE$ be a weak-* sequentially complete anti-dual pair and let $\seq{A}$ be an increasing sequence of positive operators, bounded by a positive operator $B\in\lef$: 
    \begin{equation*}
    A_n\leq A_{n+1}\leq B,\qquad n=1,2,\ldots.
    \end{equation*}
   Then $\seq{A}$ converges pointwise to a positive operator  $A\in\lef, A\leq B$ (i.e., $\dual{A_nx}{y}\to\dual{Ax}{y}$ for all $x,y\in E$).
\end{lemma}
\begin{proof}
    Since we have 
    \begin{align*}
        \abs{\dual{(A_n-A_m)x}{y}}^2&\leq\dual{(A_n-A_m)x}{x}\dual{(A_n-A_m)y}{y}\\
                                    &\leq \dual{(A_n-A_m)x}{x}\dual{By}{y}
    \end{align*}
    for every $x,y\in E$ and every integer $n\geq m$, it follows that, for every fixed $x\in E$, $(A_nx)_{\nen}$ is a weak Cauchy sequence in $F$. By weak-* sequentially completeness,  there is a vector $Ax\in F$ such that $A_nx\to Ax$ weakly. A straightforward  calculation shows that the pointwise limit $A:E\to F$  is a positive (hence weakly continuous) operator such that $A\leq B$.
\end{proof}
We are going to use this result in the following situation: let $A,B\in\lef$ be positive operators on the weak-* sequentially complete anti-dual pair $\dual FE$. Letting $A_n\coloneqq (nA):B$, we have 
\begin{equation*}
    A_n\leq A_{n+1}\leq B,\qquad n=1,2,\ldots
\end{equation*}
by Proposition \ref{P:propertiesA:B}. Lemma \ref{L:monoton} tells us that the  limit 
\begin{eqnarray}\label{E:[A]B}
\dual{[A]Bx}{y}\coloneqq \limn \dual{((nA):B)x}{y},\qquad x,y\in E
\end{eqnarray}
defines  a positive operator $[A]B\in \lef$ such that $[A]B\leq B$. Our next claim in what follows is to show that $[A]B$ coincides with the $A$-absolutely continuous part $B_a$ of $B$: 
\begin{equation*}
    [A]B=J_B^{}(I-P)J_B^*.
\end{equation*}

To establish the last claim let us introduce the following linear subspaces $\mathcal{N}_{\alpha}$ of $\hila\times\hilb$ for every positive number $\alpha>0$ as follows:
\begin{equation}\label{E:Nn}
\mathcal{N}_\alpha\coloneqq \overline{\set{(\alpha Ax,Bx)}{x\in E}}=\overline{\set{(Ax,\alpha^{-1}Bx)}{x\in E}},
\end{equation}
and denote by $Q_\alpha$ the orthogonal projection onto $\mathcal{N}_\alpha^{\perp}$. With the aid of the $Q_\alpha$'s  we provide useful factorizations for $(\alpha A):B$. 
\begin{proposition}\label{P:Q_n}
Let $A,B\in \lef$ be positive operators on the weak-* sequentially complete anti-dual pair $\dual FE$ and let $V_A,V_B$ and $Q_\alpha$ as above. Then 
\begin{equation}
(\alpha^2A):B=V_BQ_\alpha V_B^*=\alpha^2V_A^{}Q_\alpha V_A^*.
\end{equation}
\end{proposition}
\begin{proof}
For every $x\in E$ and $\alpha>0$ we have 
\begin{align*}
\dual{V_BQ_\alpha V_B^*x}{x}&=\|Q_\alpha (0,Bx)\|^2_{\hila\times\hilb}=\dist^2((0,Bx),\mathcal{N}_\alpha)\\
&=\inf\set[\big]{\big\|(0,Bx)-(Ay,\alpha^{-1} By)\big\|^2_{\hila\times\hilb}}{y\in E}\\
&=\inf\set[\big]{\dual{Ay}{y}+\dual{B(x-\alpha^{-1} y)}{x-\alpha^{-1}y}}{y\in E}\\
&=\inf\set[\big]{\dual{(\alpha^2 A)(x-z)}{x-z}+\dual{Bz}{z}}{z\in E}\\
&=\dual{((\alpha^2A):B)x}{x},
\end{align*}
which proves the first identity. The second one is proved by the same argument:  
\begin{align*}
\dual{\alpha ^2V_A^{}Q_\alpha V_A^*x}{x}&=\|Q_\alpha (\alpha Ax,0)\|^2_{\hila\times\hilb}\\ &=\dist^2((\alpha Ax,0),\mathcal{N}_\alpha)\\
&=\inf\set[\big]{\|(\alpha Ax,0)-(\alpha Ay,By)\|_{\hila\times\hilb}^2}{y\in E}\\
&=\inf\set{\dual{(\alpha^2A)(x-y)}{x-y}+\dual{By}{y}}{y\in E}\\
&=\dual{((\alpha^2A):B)x}{x},
\end{align*}
as it is claimed.
\end{proof}
The proof of identity $[A]B=B_a$ relies on the ensuing lemma:
\begin{lemma}\label{T:Q_ntoQ}
Let $A,B\in\lef$ be positive operators and let $0\oplus(I-P)$  be the orthogonal projection of $\hila\times\hilb$ onto $\{0\}\times\M^{\perp}$. Then $Q_n\to 0\oplus(I-P)$ in the strong operator topology. 
\end{lemma}
\begin{proof}
We proceed in three steps. First we prove that 
\begin{equation}\label{E:step1}
Q_n(\zeta,0)\to0, \qquad  \zeta\in \hila.
\end{equation}
To this aim take $x\in E$. Since $(n^2A):B\leq B$, we obtain by Proposition \ref{P:Q_n} that 
\begin{align*}
\|Q_n(Ax,0)\|^2=\dual{V_A^{}Q_nV_A^*x}{x}=\frac{1}{n^2}\dual{(n^2A):Bx}{x}\leq\frac{1}{n^2}\dual{Bx}{x}\to0.
\end{align*}
The sequence $(Q_n)_{\nen}$ is uniformly bounded, hence \eqref{E:step1} follows by standard density arguments. 

Our next claim is to show 
\begin{equation}\label{E:step2}
Q_n(0,\xi)=0, \qquad  \xi\in \M
\end{equation}
 for every integer $n$. First observe that  $(n^2A):B\leq n^2A$ and hence $(n^2A):B$ is $A$-absolutely continuous. We have $(n^2A):B\leq B$ on the other hand, so $(n^2A):B\leq B_a$ by Theorem \ref{T:main_Lebdecomp}. That yields
\begin{align*}
\|Q_n(0,Bx)\|^2&=\dual{(n^2A):Bx}{x}\leq \dual{B_ax}{x}\\ &=\sipb{(I-P)(Bx)}{(I-P)(Bx)}=\|(0\oplus(I- P))(0,Bx)\|^2.
\end{align*}
Hence, by density, $\|Q_n(0,\xi)\|^2\leq\|(0\oplus(I-P))(0,\xi)\|^2$ for all $\xi\in\hilb$. This implies \eqref{E:step2}.

In the final step of the proof we show that 
\begin{equation}\label{E:step3}
Q_n(0,\xi)\to(0,\xi),\qquad   \xi\in \M^{\perp}.
\end{equation}
 Since $\M^{\perp}=\overline{\dom \widehat{B}^*}$, it suffices to prove  \eqref{E:step3} for $\xi\in\dom \widehat{B}^*$ because of uniform boundedness. Consider $\xi\in\dom \widehat{B}^*$. According to Lemma \ref{L:domBkalap}  there is $m_{\xi}\geq0$ such that 
 \begin{equation*}
 \abs{\sipb{Bx}{\xi}}^2\leq m_{\xi}\dual{Ax}{x} \qquad \textrm{for all $x\in E$.}
\end{equation*}
Consequently,  
\begin{align*}
\|(I-Q_n)(0,\xi)\|^2&=\sup\set[\big]{\Abs{\sip{(nAx,Bx)}{(0,\xi)}}^2}{x\in E, \|(nAx,Bx)\|^2\leq1}\\
&=\sup\set[\big]{\abs{\sipb{Bx}{\xi}}^2}{x\in E, n^2\dual{Ax}{x}+\dual{Bx}{x}\leq1}\\
&\leq\sup\set{m_{\xi}\dual{Ax}{x}}{x\in E,n^2\dual{Ax}{x}\leq1 }\\
&\leq \frac{m_{\xi}}{n^2}\to0.
\end{align*}
The proof is complete.
\end{proof}
We can now prove the following result:
\begin{theorem}\label{T:[A]B=Ba}
Let $A,B\in\lef$ be positive operators on the weak-* sequentially complete anti-dual pair $\dual FE$, then 
\begin{equation}
\limn\dual{((nA):B)x}{y}=\dual{J_B^{}(I-P)J_B^*x}{y},\qquad x,y\in E.
\end{equation}
    In other words, $[A]B$ is identical with the $A$-absolutely continuous part $B_a$ of $B$.
\end{theorem}
\begin{proof}
By Lemma \ref{T:Q_ntoQ} we infer that 
\begin{align*}
\dual{[A]Bx}{y}&=\limn\dual{(n^2A):Bx}{y}=\limn \dual{V_BQ_nV_B^*x}{y}\\
&=\dual{V_B(0\oplus(I-P))V_B^*x}{y}=\sip{(0\oplus(I-P))(0,Bx)}{(0,By)}\\
&=\sipb{(I-P)(Bx)}{By}=\dual{J_B^{}(I-P)J_B^*x}{y},
\end{align*}
that is, $[A]B=B_a$.
\end{proof}
\section{Characterizations of absolute continuity}
It is clear from Theorem \ref{T:main_Lebdecomp} that a positive operator $B$ is absolutely continuous with respect to the positive operator $A$ if and only if $B$ is identical with its $A$-absolutely continuous part $B_a$. In the light of this, Theorem \ref{T:[A]B=Ba} yields yet another characterization absolute continuity, namely, $B\ll A$ if and only if 
\begin{equation*}
    B=[A]B.
\end{equation*}
In particular, every $A$-absolutely continuous operator $B$ on a weak-* sequentially complete anti-dual pair  can be obtained as the pointwise limit of a monotone increasing  sequence $\seq B$  such that $B_n\leq \alpha_n A$ for some nonnegative sequence $\seq \alpha$. 
For positive operators on a Hilbert space,  Ando \cite{Ando} introduced the concept of being absolutely continuous just by this property. Adopting the rather expressive terminology of \cite{Hassi2009}, such a positive operator $B$ will be called ``almost dominated'' by $A$.  

In the first result of this section we are going to show that almost dominated operators are just the absolutely continuous ones: 
\begin{theorem}\label{T:almostdom=ac}
Let $A,B$ be positive operators on the weak-* sequentially anti-dual pair $\dual FE$. The following conditions are equivalent:
 \begin{enumerate}[label=\textup{(\roman*)}, labelindent=\parindent]
\item $B$ is absolutely continuous with respect to $A$.
\item $B$ is almost dominated by $A$, that is, there exists a monotone increasing sequence $(B_n)_{\nen}$ of positive operators in $\lef$ and $(\alpha_n)_{\nen}$ of positive numbers such that  $B_n\leq \alpha_n A$ and $B_n\to B$ pointwise on $E$. 
\end{enumerate}
\end{theorem}
\begin{proof}
Implication (i)$\Rightarrow$(ii) is  clear from Theorem \ref{T:[A]B=Ba} and from what has been said above. For the converse implication let $(B_n)_{\nen}$ be a sequence  satisfying (ii). By Corollary \ref{C:factor}, for any integer $n$ there is a positive operator $C_n\in\mathscr{B}(\hilb), \|C_n\|\leq1$,  such that $B_n=J_BC_nJ_B^*$. We claim that 
\begin{eqnarray}\label{E:ranCnJ}
\ran C_n\subseteq \dom \widehat B^*,\qquad n=1,2,\ldots.
\end{eqnarray}
For let $\xi\in \hilb$, then for every $x\in E$ we  have
\begin{align*}
    \abs{\dual{J_BC_n\xi}{x}}^2&=\abs{\sipb{\xi}{C_nJ_B^*x}}^2\\
    &\leq \|\xi\|^2_B \|C_nJ_B^*x\|^2_B\\
    &\leq \|\xi\|^2_B \|C^{1/2}_nJ_B^*x\|^2_B\\
    &=\|\xi\|^2_B \dual{B_nx}{x}\\
    &=\alpha_n\|\xi\|^2_B\dual{Ax}{x}\\
    &=\alpha_n\|\xi\|^2_B\|J_A^*x\|^2_A,
\end{align*}
whence we  conclude that $J_BC_n\xi\in \ran J_A$ by Lemma \ref{L:adjointrange}. By Lemma \ref{L:domBkalap}, $C_n\xi\in\dom\widehat{B}^*$, that proves \eqref{E:ranCnJ}. Since $B$ is absolutely continuous precisely when $\widehat{B}^*$ is densely defined, it suffices to prove that the union of ranges of the $C_n$'s is dense in $\hilb$, or equivalently, 
\begin{equation}\label{E:kerCn}
  \bigcap_{n=1}^\infty \ker C_n=\{0\}. 
\end{equation}
To check this identity take  $x\in E$, then 
\begin{align*}
    \|Bx-C_n(Bx)\|^2_B&=\dual{Bx}{x}-2\sipb{C_nJ_B^*x}{J_B^*x}+\|C_n(Bx)\|^2_B\\
    &\leq 2\dual{Bx}{x}-2\dual{B_nx}{x}\to0,
\end{align*}
from which we deduce that $C_n$ converges strongly to the identity operator of $\hilb$, and this clearly implies \eqref{E:kerCn}.
\end{proof}

We remark that the existence of a Lebesgue-type decomposition can be proved easily by means of (ii) with an elementary iteration involving parallel addition (see  \cite{Arlinskii} and \cite{Arlinskii-T}). Hovewer, the itaration itself  does not guarantee the maximality of the resulted absolutely continuous part.

In the rest of the section, our goal is to give a Radon--Nikodym type characterization of absolute continuity. In order to formulate our main result we need some  preliminaries. 

Let $\dual FE$ be a weak-* sequentially complete dual pair and consider two positive operators $A,B\in\lef$ on it. Denote by $\hilab$ the corresponding auxiliary Hilbert space, and by $J\coloneqq J_{A+B}$ the natural embedding operator of $\hilab$ into $F$. A straightforward application of Lemma \ref{C:factor} gives then two positive contractions $C_A, C_B\in \mathscr{B}(\hilab)$ such that 
\begin{equation*}
    \sipab{C_AJ^*x}{J^*y}=\dual{Ax}{y}, \qquad \sipab{C_BJ^*x}{J^*y}=\dual{Bx}{y}
\end{equation*}
for every $x,y\in E$.  Our first technical lemma gives some characterizations of the absolute continuity in terms of $C_A$ and $C_B$:
\begin{lemma}\label{L:kerT}
The following assertions are equivalent:
\begin{enumerate}[label=\textup{(\roman*)}, labelindent=\parindent]
\item $B$ is absolutely continuous with respect to $A$,
\item $\ker C_A=\{0\}$,
\item $\ran C_B\subseteq (\ker C_A)^{\perp}.$
\end{enumerate}
\end{lemma}
\begin{proof}
Assume first that $B$ is $A$-absolutely continuous. If $\xi\in\ker C_A$, then there exists a sequence $\seq{x}$ in $E$ so that $J^*x_n\to\xi$ in $\hilab$ and $\dual{Ax_n}{x_n}\to0$. It is clear  that  $\dual{B(x_n-x_m)}{x_n-x_m}\to0$, hence
\begin{align*}
\|\xi\|_{A+B}^2=\limn \dual{(A+B)x_n}{x_n}=\limn\dual{Bx_n}{x_n}=0,
\end{align*}
because of absolute continuity. Thus (i) implies (ii). The converse implication goes similar: assume that $\ker C_A=\{0\}$ and consider a sequence $\seq{x}$ so that $\dual{Ax_n}{x_n}\to0$ and that $\dual{B(x_n-x_m)}{x_n-x_m}\to0$. Clearly, $\seq{(A+B)x}$ is  a Cauchy sequence in $\hilab$ and its limit $\xi$ belongs to $\ker C_A=\{0\}$. This yields 
\begin{equation*}
\dual{Bx_n}{x_n}\leq \dual{(A+B)x_n}{x_n}\to\|\xi\|^2_{A+B}=0,
\end{equation*}
thus $B\ll A$. 
It remains to show that  (iii) implies (ii) (the backward implication being trivial). That will follow apparently if we show that 
\begin{equation}\label{E:kerTb}
\ker C_B\subseteq (\ker C_A)^{\perp}
\end{equation}
holds for arbitrary $A$ and $B$. To this end, take $\xi\in \ker C_A$, $\zeta\in\ker C_B$, and choose a sequence $\seq{x}$  such that 
\begin{align*}
J^*x_n\to &\xi,\qquad\mbox{and}\qquad \dual{Ax_n}{x_n}\to0,
\end{align*}
and choose another sequence $\seq{y}$ such that 
\begin{align*}
J^*y_n\to &\zeta,\qquad\mbox{and}\qquad \dual{By_n}{y_n}\to0.
\end{align*}
Both the sequences $\dual{Bx_n}{x_n}$ and $\dual{Ay_n}{y_n}$ are bounded, hence  
\begin{align*}
\abs{\sipab{\xi}{\zeta}}=&\limn \abs{\sipab{J^*x_n}{J^*y_n}}\\
                        =&\limn \abs{\dual{Ax_n}{y_n}+\dual{Bx_n}{y_n}}\\ 
                        \leq& \limsup_{n\to\infty} \dual{Ax_n}{x_n}^{1/2}\dual{Ay_n}{y_n}^{1/2}\\ 
                        &+ \limsup_{n\to\infty} \dual{Bx_n}{x_n}^{1/2}\dual{By_n}{y_n}^{1/2}=0,
\end{align*}
which proves \eqref{E:kerTb}.
\end{proof}
Now we can prove the main result of this section:
\begin{theorem}\label{T:R-N}
Let $\dual FE$ be a weak-* sequentially complete anti-dual pair  let $A,B\in\lef$ be positive operators. The following are equivalent:
 \begin{enumerate}[label=\textup{(\roman*)}, labelindent=\parindent]
\item $B$ is absolutely continuous with respect to $A$,
\item for every $y\in E$ there exists a sequence $\seq{y}$ in $E$ such that 
\begin{equation*}
\dual{Bx}{y}=\limn \dual{Ax}{y_n},\qquad x\in E,
\end{equation*}
and the convergence is uniform on the set $\set{x\in E}{\dual{(A+B)x}{x}\leq1}$.
\end{enumerate}
\end{theorem}
\begin{proof}
Recall that $B$ is $A$-absolutely continuous precisely if $\widehat{B}$ is (the graph of) a closable operator, or equivalently, if $\dom \widehat{B}^*$ is dense in $\hilb$. Hence, if $B\ll A$ then   for every vector $\xi\in\hilb$ there is a sequence $\seq{\xi}$ in  $\dom \widehat{B}^*$ such that $\|\xi-\xi_n\|_B\leq 1/n$, $n=1,2,\ldots$ Furthermore, by density, we can find a sequence $\seq{y}$ in $E$ such that $\|\widehat{B}^*\xi_n-Ay_n\|_A\leq 1/n$ for each $n$. Hence 
\begin{align*}
\abs{\sipb{Bx}{\xi}- \dual{Ax}{y_n}}&\leq \abs{\sipb{Bx}{\xi-\xi_n}}+\abs{\sipb{Bx}{\xi_n}-\sipa{Ax}{Ay_n}}\\
&= \abs{\sipb{Bx}{\xi-\xi_n}}+\abs{\sipa{Ax}{\widehat{B}^*\xi_n}-\sipa{Ax}{Ay_n}}\\
&\leq \frac1n (\|Ax\|_A+\|Bx\|_B)\leq \frac{\sqrt{2}}{n} \sqrt{\dual{(A+B)x}{x}}.
\end{align*}
The choice $\xi\coloneqq By\in\hilb$ with an arbitrary $y\in E$ gives that (i) implies (ii).  Let us prove now the backward implication: let $y\in E$ and choose $\seq{y}$ according to (ii). Since $\ran J^*$ is dense in $\hilab$, it follows that 
\begin{align*}
\|C_B&J^*y-C_AJ^*y_n\|_{A+B}\\
&=\sup_{x\in E, \|J^*x\|_{A+B}\leq1} \abs{\sipab{J^*x}{C_BJ^*y-C_AJ^*y_n}}\\
&=\sup_{x\in E, \dual{(A+B)x}{x}\leq1}\abs{\dual{Bx}{y}- \dual{Ax}{y_n}}\to0.
\end{align*}
We see therefore that  $\ran C_BJ^*\subseteq \overline{\ran C_A}$, and hence  $ \ran C_B\subseteq(\ker C_A)^{\perp}$. Lemma \ref{L:kerT} completes the proof. 
\end{proof}
\section{Characterizations of singularity}
This section is devoted to some characterizations of singularity. Note that the original definition of singularity is rather algebraic as depending on the ordering induced by positivity.  Below we are going to provide some further equivalent characterizations which reflect some geometric and metric features of singularity.
For analogous results see  \cites{Ando,Hassi2007,Tarcsay_parallel}.
\begin{theorem}\label{T:charsing}
Let $\dual FE$ be a weak-* sequentially complete anti-dual pair and let $A,B\in\lef$ be positive operators on it. The following assertions are equivalent: 
\begin{enumerate}[label=\textup{(\roman*)}, labelindent=\parindent]
\item $A$ and $B$ are mutually singular,
\item $A:B=0,$
\item the set $\set{(Ax,Bx)}{x\in E}$ is dense in $\hila\times\hilb$,
\item $\xi=0$ is the only vector in $\hilb$ such that $\abs{\sipb{Bx}{\xi}}^2\leq M_{\xi}\dual{Ax}{x}$ for every $x$ in $E$,
\item $\M=\hilb,$
\item $\ran J_A\cap\ran J_B=\{0\}$, 
\item for every $x$ in $E$ there is a sequence $\seq x$  such that
\begin{equation*}
    \dual{Ax_n}{x_n}\to0\qquad \mbox{and}\qquad \dual{B(x-x_n)}{x-x_n}\to0.
\end{equation*}
\end{enumerate}
\end{theorem}
\begin{proof}
Since $A:B\leq A$ and $A:B\leq B$, (i) implies (ii). Assume (ii), then
\begin{equation*}
A:B=V_A^{}QV_A^*=V_B^{}QV_B^*=0,
\end{equation*}
where $Q$ is the orthogonal projection of $\hila\times \hilb$ onto $\widehat{B}^{\perp}$.
That gives $Q(Ax,0)=0=Q(0,By)$ for every  $x,y\in E$. Since $\ran A\times \ran B$ is dense in $\hila\times\hilb$ it follows that 
\begin{equation*}
\{0\}=\ran Q=\set{(Ax,Bx)}{x\in E}^{\perp},
\end{equation*}
hence (ii) implies (iii). That (iii) implies (ii) is clear from identity $A:B=V_A^{}QV_A^*$. Observe furthermore that 
\begin{equation*}
A:B\leq (nA):B\leq (nA):(nB)=n(A:B),
\end{equation*}
hence   $A:B=0$ implies $(nA):B=0$ for each $n$, and  we have therefore  $B_a=[A]B=0$ by Theorem \ref{T:[A]B=Ba}. This means that $B=B_s$ in the view of Theorem \ref{T:main_Lebdecomp}, hence $B\perp A$. This proves that (ii) implies (i). The equivalence between (iv), (v), (vi) is clear from Lemma \ref{L:domBkalap} and identity \eqref{E:domBperp=M}. Supposing (v) we have $B_a=0$ and hence $B_s=B$ and hence $B\perp A$  by Theorem \ref{T:main_Lebdecomp}. Conversely, if $B$ and $A$ are mutually singular then, as it has been shown above, $(nA):B=0$ for each $n$ and thus $J_B^{}(I-P)J_B^*=[A]B=0$. Consequently, $(I-P)(Bx)=0$ for every  $x\in E$ and therefore $I-P=0$ by density. Hence (i) implies (v). We see therefore that (i)-(vi) are equivalent. Finally, suppose (iii) and fix $x\in E$, then there is a sequence $\seq x$ such that $(Ax_n,Bx_n)\to (0,Bx)$ in $\hila\times\hilb$, which clearly implies $\dual{Ax_n}{x_n}\to 0$ and $\dual{B(x_n-x_m)}{x_n-x_m}\to 0$, hence (iii) implies (vii). Conversely, if we assume (vii) then we have that the dense set $\ran B$ is included in $\M$, hence $\M=\hilb$. This means that (vii) implies (v). 
\end{proof}
As an immediate consequence we conclude that absolute continuity and singularity are complementary notions in some sense: 
\begin{corollary}
Let $A$ be a positive operator on the weak-* sequentially complete anti-dual pair $\dual FE$. Then $B=0$ is the unique positive operator which is simultaneously $A$-absolutely continuous and $A$-singular. In other words, $B\ll A$ and $B\perp A$ imply $B=0$.
\end{corollary}
\section{Uniqueness of the decomposition}

It was pointed out by Ando \cite{Ando} that the Lebesgue decomposition among positive operators on an infinite dimensional Hilbert space is not unique. Since anti-dual pairs are even more general, we expect the same in our case. The reason why non-uniqueness occurs in the non-commutative integration theory is that absolute continuity is not hereditary: $B\ll A$ and $C\leq B$ do not imply $C\ll A$. In fact, it may even happen that $C\neq 0$ and $C\perp A$. More explicitly,  we have the following result:

\begin{proposition}\label{P:7.1}
Let $A,B$ be positive operators on the weak-* sequentially anti-dual pair $\dual FE$. Suppose that $B$ is $A$-absolutely continuous but not $A$-dominated, i.e., there is no $\alpha\geq 0$ such that $B\leq \alpha A$. Then there is a non-zero positive operator $B'\leq B$ such that $B'\perp A$. 
\end{proposition}
 \begin{proof}
By assumption, $\widehat{B}$ is a densely defined closed and unbounded operator between $\hila$ and $\hilb$, hence $\dom \widehat{B}^*$ is a proper dense subspace of $\hilb$. Choose a vector $\zeta\in\hilb\setminus \dom \widehat{B}^*$ and denote by $Q_\zeta$ the orthogonal projection onto the one-dimensional subspace $H_\zeta$ generated by $\zeta$. Set $B'\coloneqq J_BQ_\zeta J_B^*$, then  clearly $B'\leq B$. We claim that $B'\perp A$, which is equivalent to $\ran J_A\cap \ran J_{B'}=\{0\}$ by Theorem \ref{T:charsing}. To see this we observe first that $\ran J_{B'}=\ran J_BQ_\zeta$ because of  Theorem \ref{T:Douglas}, namely, 
\begin{equation*}
    \|J_{B'}^*x\|^2_{B'}=\dual{B'x}{x}=\dual{J_BQ_\zeta J_B^*x}{x}=\|Q_\zeta J_B^*x\|_{B}^2,\qquad x\in E.
\end{equation*}
Suppose $f=J_BQ_\zeta\xi$ belongs to $\ran J_A$, which means that $Q_\zeta\xi\in \dom\widehat{B}^*$, according to Lemma \ref{L:domBkalap}. Since we have $H_\zeta\cap \dom \widehat{B}^*=\{0\}$, it follows that $Q_\zeta\xi=0$ and $f=0$, accordingly.
\end{proof}
The next result gives a complete characterization of uniqueness of the Lebesgue decomposition. We mention that this is a direct generalization of Ando's uniqueness result \cite{Ando}*{Theorem 6}. We  also refer the reader to \cite{Hassilebesgue2018}*{Theorem 7.8} and \cite{Hassi2009}*{Theorem 4.6}.
\begin{theorem}\label{T:lebunique}
Let $\dual FE $ be a weak-* sequentially complete anti-dual pair and let $A,B\in \lef$ be positive operators.  The following statements are equivalent:
\begin{enumerate}[label=\textup{(\roman*)}, labelindent=\parindent]
\item the Lebesgue-decomposition of $B$ into $A$-absolutely continuous and $A$-singul\-ar parts  is unique,
\item $\dom \widehat{B}^*\subseteq \hilb$ is closed,
\item $\widehat{B}_{\textup{reg}}$ is norm continuous between $\hila$ and $\hilb$,
\item $B_a\leq \alpha A$ for some $\alpha\geq0$,
\item $J_B\langle \M^{\perp}\rangle\subseteq \ran J_A$.
\end{enumerate}
\end{theorem}
\begin{proof}
We start by proving that (i) implies (ii). Suppose therefore that $\dom \widehat B^*$ is not closed and consider a unit vector $\zeta \in\M^{\perp}\setminus \dom\widehat{B}^*$. Denote by $Q_\zeta$ the orthogonal projection onto the one dimensional subspace $\hil_\zeta$ spanned by $\zeta$. Then $P_1\coloneqq P+Q_\zeta$ is a projection and 
\begin{equation*}
    B_1\coloneqq J_B^{}(I-P_1)J_B^*,\qquad B_2\coloneqq J_BP_1J_B^*
\end{equation*}
are positive operators from $E$ to $F$ such that $B_1+B_2=B$. Clearly, $B_1\neq B_a$ and $B_2\neq B_s$. We claim that $B_1\ll A$ and $B_2\perp A$. Since the map 
\begin{equation*}
    T(Ax)\coloneqq (I-P)Bx,\qquad x\in E
\end{equation*}
defines a closable operator between $\hila$ and $\hilb$, it follows that $(I-Q_\zeta)T$ is closable too. Indeed, 
\begin{equation*}
    \dom(T^*(I-Q_\zeta))=H_\zeta\oplus [\dom T^*\cap H_\zeta^\perp],
\end{equation*}
and it is known that a dense subspace of a Hilbert space is dense in every finite co-dimensional subspace. Consequently, $((I-Q_\zeta)T)^*$ is densely defined and hence $(I-Q_\zeta)T$ is closable. A straightforward calculation shows that \begin{equation*}
    \|(I-Q_\zeta)T(Ax)\|_B^2=\dual{B_1x}{x},\qquad x\in E,
\end{equation*}
whence it follows that $B_1\ll A$. To check that  $B_2\perp A$ we argue as in the proof of Proposition \ref{P:7.1}. First we observe that $\ran J_{B_2}=\ran J_B^{}(P+Q_\zeta)$ by Theorem \ref{T:Douglas}. Furthermore, if $f=J_B^{}(P+Q_\zeta)\xi\in\ran J_A$ then $(P+Q_\zeta)\xi\in\dom \widehat{B}^*$, according to Lemma \ref{L:domBkalap}. But we have $\ran (P+Q_\zeta)\cap \dom \widehat{B}^*=\{0\}$, hence $f=0$. Consequently, $\ran J_A\cap \ran J_{B_2}=\{0\}$, and therefore $B_2\perp A$. Summing up, $B=B_1+B_2$ is a Lebesgue decomposition of $B$ with respect to $A$ that differs from the canonical Lebesgue decomposition $B=B_a+B_s$, i.e., the Lebesgue decomposition is not unique. 

To prove that (ii) implies (iii) assume that $\dom \widehat{B}^*$ is a closed subspace of $\hilb$, then $\widehat{B}^*$ is bounded by the  closed graph theorem. The same holds true for $\widehat{B}_{\textup{reg}}$. If $\widehat{B}_{\textup{reg}}$ is bounded, then from \eqref{E:BJABJA} we conclude that $B_a=J_A\widehat{B}_{\textup{reg}}^*\widehat{B}_{\textup{reg}}J_A^*$, and therefore $B_a\leq \alpha A$ with $\alpha\coloneqq \|\widehat{B}_{\textup{reg}}\|^2$. Hence (iii) implies (iv). Note that (iv) is equivalent to $\ran J_B^{}(I-P)\subseteq \ran J_A$ in virtue of Theorem \ref{T:Douglas}, hence (iv) and (v) are equivalent. Assume finally (iv) and let $B_1+B_2$ be any Lebesgue decomposition of $B$ with respect to $A$, where $B_1\ll A$ and $B_2\perp A$. By Theorem \ref{T:main_Lebdecomp}, $B_2\leq B_a$, hence $0\leq B_a-B_1\leq B_a\leq \alpha A$. Consequently, $0\leq B_2-B_s=(B_a-B_1)\leq \alpha A$ and $B_2-B_s\leq B_2$, and therefore $B_2-B_s=0$ by singularity. This means that $B_2=B_s$ and $B_1=B_0$, proving that the Lebesgue decomposition is unique. The proof is complete.     
\end{proof}

Below we give a sufficient condition on an operator $A$ such that the $A$-Lebesgue decomposition of every operator $B$ be unique. 
\begin{lemma}
Let  $A$  be a  positive operator on a weak-* sequentially complete anti-dual pair $\dual FE$.  The following assertions are equivent:
\begin{enumerate}[label=\textup{(\roman*)}, labelindent=\parindent]
  \item  $\ran A$ is weak-* sequentially closed in $F$,
  \item $\ran A$ is a Hilbert space under the inner product $\sipa{\cdot}{\cdot}$.
\end{enumerate}
\end{lemma}
\begin{proof}
Assume first that  $\ran A$ is weak-* sequentially closed in $F$. We are going to show that $\hila=\ran A$. It suffices to show that $J_A$ coincides with its restriction to $\ran A$. That will be obtained by showing that $\ker J_A\subseteq\ker (J_A|_{\ran A})$ and $\ran J_A\subseteq\ran (J_A|_{\ran A})$. The first inclusion is clear because $J_A$ is injective:
\begin{equation*}
\ker J_A=(\ran J_A^*)^{\perp}=(\ran A)^{\perp}=\{0\}.
\end{equation*}
The second range inclusion follows from the fact that $\ran J_A$ is contained in the weak-* sequential closure of the range of $J_A|_{\ran A}$ in $F$, that is identical with  $\ran A$ by (i). This proves that (i) implies (ii). Assume conversely that $\ran A=\hila$ and let $f$ belong to the weak-* sequential closure of $\ran A$ in $F$. Choose a sequence $\seq{x}$ such that 
\begin{eqnarray*}
\dual{f}{x}=\limn \dual{Ax_n}{x},\qquad\textrm{$x\in E$}.
\end{eqnarray*}
For every  $n$ let us define a continuous conjugate linear functional $\phi_n:\hila\to\dupC$ by 
\begin{eqnarray*}
\phi_n(Ax)\coloneqq \sipa{Ax_n}{Ax}=\dual{Ax_n}{x}, \qquad x\in E.
\end{eqnarray*}
Then $\seq \phi$ converges pointwise to some bounded conjugate linear functional  $\phi:\hila\to\dupC$, because of the Banach--Steinhaus theorem. By the Riesz representation theorem, there exists $z\in E$ such that   $\phi(Ax)=\sipa{Az}{Ax}$, $x\in E$, and therefore
\begin{eqnarray*}
\dual{f}{x}=\sipa{Az}{Ax}=\dual{Az}{x},\qquad x\in E.
\end{eqnarray*}
Consequently, $f=Az\in\ran A$.
\end{proof}
\begin{theorem}
Let  $A$  be a  positive operator on a weak-* sequentially complete anti-dual pair $\dual FE$. If the range of $A$ is weak-* sequentially closed then every  positive operator $B\in\lef$ admits a unique Lebesgue decomposition with respect to $A$.
\end{theorem}
\begin{proof}
 According to the preceding lemma,  $\ran A$ is complete under the inner product $\sipa\cdot\cdot$, i.e., $\ran A=\hila$. The closed operator $\widehat{B}_{\textup{reg}}$ is everywhere defined on $\hila$ and thus bounded by the closed graph theorem. The Lebesgue decomposition of $B$ with respect to $A$ is unique by Theorem \ref{T:lebunique}. 
\end{proof}
 
 \section{Applications}
  To conclude the paper we apply the developed decomposition theory to some concrete objects including  Hilbert space operators, Hermitian forms, representable functionals, and additive set functions.
 \subsection{Positive operators on Hilbert spaces}
 Let $\hil$ be a complex Hilbert space with inner product $\sip\cdot\cdot$, then $\dual\hil\hil$ forms a anti-dual pair with $\dual{\cdot}{\cdot}\coloneqq \sip{\cdot}{\cdot}$. An immediate application of the Banach--Steinhaus theorem shows that $\dual\hil\hil$ is weak-* sequentially complete, thus everything what has been said so far remains valid for $\dual{\hil}{\hil}$ and the positive operators on it. 
 
 We shortly summarize Ando's main results \cite{Ando}*{Theorem 2 and 6} in a statement. The proof follows immediately from Theorem \ref{T:main_Lebdecomp}, \ref{T:almostdom=ac} and \ref{T:lebunique}.
 \begin{theorem}
 Let $A,B$ be bounded positive operators on a complex Hilbert space $\hil$ and let $B_a\coloneqq \limn (nA):B$ where the limit is taken in the strong operator topology and let $B_s\coloneqq B-B_a$. Then
 \begin{equation}\label{E:Andocorollary}
     B=B_a+B_s
 \end{equation}
is a Lebesgue-type decomposition, i.e.,  $B_a$ is $A$-absolutely continuous and  $B_s$ is $A$-singular.  $B_a$ is maximal among those positive operators $C\geq0$ such that $C\leq B$ and $C\ll A$. The Lebesgue decomposition \eqref{E:Andocorollary} is unique if and only if $B_a\leq \alpha A$ for some constant $\alpha\geq0$.  
 \end{theorem}
\subsection{Nonnegative forms.} Let $\mathfrak{D}$ be a complex vector space and let $\tform$, $\wform$ be nonnegative Hermitian forms on it. Let us denote by $\bar{\mathfrak{D}}^*$ the algebraic dual space of $\mathfrak{D}$, then  $\dual{\bar{\mathfrak{D}}^*}{\mathfrak D}$ forms a weak-* sequentially complete anti-dual pair and 
\begin{equation*}
    \dual{Tx}{y}\coloneqq \tform(x,y),\qquad \dual{Wx}{y}\coloneqq \wform(x,y),\qquad x,y\in\mathfrak D
\end{equation*}
define two positive operators $T,W:\mathfrak D\to \bar{\mathfrak{D}}^*$. We recall that the form $\tform$ is called $\wform$-almost dominated if there is a  monotonically nondecreasing sequence of forms $\tform_n$ such that $\tform_n\leq \alpha_n \wform$ for some $\alpha_n\geq0$ and $\tform_n\to\tform$ pointwise. Similarily, $\tform$ is called $\wform$-closable if for every sequence $\seq x$ of $\mathfrak{D}$ such that $\wform(x_n,x_n)\to0$ and $\tform(x_n-x_m,x_n-x_m)\to0$ it follows that $\tform(x_n,x_n)\to0$. 

It is immediate to conclude that the form $\tform$ is $\wform$-closable if and only if the operator $T$ is $W$-absolutely continuous. Similarly, $\tform$ is $\wform$-almost dominated precisely  when $T$ is $W$-almost dominated. Consequently, from Theorem \ref{T:almostdom=ac} it follows that the notions of closability and almost dominatedness are equivalent (cf. also \cite{Hassi2009}*{Theorem 3.8}). The map $\tform\mapsto T$ between nonnegative hermiatian forms and positive operators on $\mathfrak{D}$ is a bijection, so from Theorem \ref{T:main_Lebdecomp} and \ref{T:lebunique} we conclude the following result (see \cite{Hassi2009}*{Theorem 2.11 and 4.6}):
\begin{theorem}
 Let $\tform, \wform$ be nonnegative Hermitian forms on a complex vector space $\mathfrak{D}$ and let $\tform_a(x,x)\coloneqq \limn ((n\tform):\sform)(x,x)$, $ x\in\mathfrak{ D}$ and  $\tform_s\coloneqq \tform-\tform_a$. Then
 \begin{equation}\label{E:hassicorollary}
     \tform=\tform_a+\tform_s
 \end{equation}
is a Lebesgue-type decomposition of $\tform$ with respect to $\wfrom$, i.e.,  $\tform_a$ is $\wform$-absolutely continuous and  $\tform_s$ is $\wform$-singular. Furthermore, $\tform_a$ is maximal among those forms $\sform$ such that $\sform\leq \tform$ and $\sform\ll \wform$. The Lebesgue decomposition \eqref{E:hassicorollary} is unique if and only if $\tform_a\leq \alpha \wform$ for some constant $\alpha\geq0$. \end{theorem}
\subsection{Representable functionals.} Let  $\alg$ be a $^*$-algebra (with or without unit), i.e., an algebra endowed with an involution. A functional $f:\alg\to\dupC$ is called representable if there is a triple $(\hil_f,\pi_f,\zeta_f)$ such that $\hil_f$ is a Hilbert space, $\zeta_f\in\hil_f$ and $\pi_f:\alg\to\balg(\hil_f)$ is a *-algebra homomorphism such that
\begin{equation*}
    f(a)=\sipf{\pi_f(a)\zeta_f}{\zeta_f},\qquad a\in\alg.
\end{equation*}
A straightforward verification shows that every representable functional $f$ is positive hence the map $A:\alg\to \bar \alg^*$ defined by 
\begin{equation}
    \dual{Aa}{b}\coloneqq f(b^*a),\qquad a,b\in\alg
\end{equation}
is a positive operator. (Note however that not every positive operator $A$ arises from a representable functional $f$ in the above way.)
Denote by $\hila$ the corresponding auxiliary Hilbert space. It is easy to show that $\pi_f:\alg\to\bha$, $a\mapsto \pi_f(a)$ is a *-homomorphism, where the bounded operator $\pi_f(a)$ arises from the densely defined one given by
\begin{equation*}
    \pi_f(a)(Ab)\coloneqq A(ab),\qquad b\in\alg.
\end{equation*}
It follows from the representability of $f$ that $\abs{f(a)}^2\leq Cf(a^*a)$, $a\in\alg$, for some constant $C\geq0$ and hence 
\begin{equation*}
    Aa\mapsto f(a),\qquad a\in\alg
\end{equation*}
defines a continuous linear functional from $\ran A\subseteq \hila$ to $\dupC$. The corresponding representing functional $\zeta_f$ satisfies
\begin{equation*}
    \sipa{Aa}{\zeta_f}=f(a),\qquad a\in\alg,
\end{equation*}
and admits the useful property $\pi_f(a)\zeta_f=Aa$. It follows therefore that 
\begin{equation*}
    f(a)=\sipa{\pi_f(a)\zeta_f}{\zeta_f},\qquad a\in\alg.
\end{equation*}
Let $g$ be another representable functional on $\alg$. We say that $g$ is $f$-absolutely continuous if   for every sequence $\seq a$ of $\alg$ such that $f(a_n^*a_n)\to0$ and $g((a_n-a_m)^*(a_n-a_m))\to0$ it follows that $g(a_n^*a_n)\to0$. Furthermore, $g$ and $f$ are singular with respect to each other if $h=0$ is the only representable functional such that $h\leq f$ and $h\leq g$. 

Denote by $B:\alg\to\bar\alg^*$ be  the positive operator associated with $g$ and let $(\hilb,\pi_g,\zeta_g)$ the corresponding GNS-triplet obtained along the above procedure. Let us introduce $\M\subseteq \hilb$ and $P$ as in Section 3. Then $\M$ and $\M^\perp$ are both  $\pi_g$-invariant, so 
\begin{equation*}
    g_s(a)\coloneqq \sipb{\pi_g(a)P\zeta_g}{P\zeta_g},\qquad  g_a(a)\coloneqq \sipb{\pi_g(a)(I-P)\zeta_g}{(I-P)\zeta_g} 
\end{equation*}
are  representable functionals on $\alg$ such that 
\begin{equation}\label{E:8.4}
    \dual{B_aa}{b}=g_a(b^*a),\qquad \dual{B_sa}{b}=g_s(b^*a).
\end{equation}
It is clear therefore that $g_a\ll f$ and $g_s\perp f$. If $\alg$ has a unit element $1$ then the absolutely continuous and singular parts can be written in a much simpler form: 
\begin{equation*}
    g_a(a)=\overline{\dual{B_a1}{a}},\qquad g_s(a)=\overline{\dual{B_s1}{a}},\qquad a\in\alg.
\end{equation*}
After these observations we can state the corresponding Lebesgue decomposition theorem of representable functionals \cite{Tarcsay_repr}*{Theorem 3.3}; cf. also \cite{Gudder}*{Corollary 3} and \cite{TZSTT-Glasgow}*{Theorem 3.3}:
\begin{theorem}\label{T:representLebesgue}
 Let $f,g$ be representable functionals on the *-algebra $\alg$, then $g_a$ and $g_s$ are representable functionals such that $g=g_a+g_s$, where $g_a$ is $f$-absolutely continuous and $g_s$ is $f$-singular. Furthermore, $g_a$ is is maximal among those representable functionals $h$ such that $h\leq g$ and $h\ll f$.
\end{theorem} 
Finally we note that not every positive operator $A:\alg\to\bar\alg^*$ arises from a representable functional, hence the question of uniqueness of the Lebesgue decomposition cannot be answered via Theorem \ref{T:lebunique}. For a detailed discussion of this delicate problem we refer the reader to \cite{TZSTT-Glasgow}.
\subsection{Finitely additive and $\sigma$-additive set functions} Let $X$ be a non-empty set and $\ring$ be an algebra of sets on $X$. Let $\alpha$ be a non-negative finitely additive measure and denote by $\step$ the unital *-algebra of $\ring$-measurable functions, then $\alpha$ induces a positive operator $A:\step\to\bar\step^*$ by
\begin{equation*}
    \dual{A\phi}{\psi}\coloneqq\int\phi\bar\psi\,d\alpha,\qquad \phi,\psi\in\step.
\end{equation*}
We notice that we can easily recover $\alpha$ from $A$, namely 
\begin{equation}\label{E:alphaR}
    \alpha(R)=\dual{A\chi^{}_R}{\chi^{}_R},\qquad R\in \ring.
\end{equation}
However, not every positive operator $A:\step\to\bar\step^*$ induces a finitely additive measure, as it turns out from the next statement.
\begin{proposition}\label{P:8.4}
If $A:\step\to\bar\step^*$ is a positive operator then \eqref{E:alphaR} defines an additive set function if and only if 
\begin{equation}\label{E:A-alpha}
    \dual{A\abs\phi}{\abs\phi}=\dual{A\phi}{\phi},\qquad \phi\in\step.
\end{equation}
\end{proposition}
\begin{proof}
 The ``only if'' part of the statement is clear. For the converse suppose that $A$ satisfies \eqref{E:A-alpha}. For every two disjoint sets $R_1,R_2\in\ring$ we have
 \begin{align*}
     \alpha(R_1\cup R_2)&=\frac12 \big\{\dual{A(\chi^{}_{R_1}+\chi^{}_{R_2})}{\chi^{}_{R_1}+\chi^{}_{R_2}}+\dual{A(\chi^{}_{R_1}-\chi^{}_{R_2})}{\chi^{}_{R_1}-\chi^{}_{R_2}}\big\}\\
     &=\dual{A\chi^{}_{R_1}}{\chi^{}_{R_1}}+\dual{A\chi^{}_{R_2}}{\chi^{}_{R_2}}=\alpha(R_1)+\alpha(R_2),
 \end{align*}
 proving the additivity of $\alpha$.
\end{proof}
Assume  that we are given another nonnegative additive set function $\beta$ on $\ring$, then $\beta$ is called absolutely continuous with respect to $\alpha$ if for each $\varepsilon>0$ there exists some $\delta>0$ such that $R\in\ring$ and $\alpha(R)<\delta$ imply $\beta(R)<\varepsilon$. Furthermore, $\alpha$ and $\beta$ are mutually singular if $\gamma=0$ is the only nonnegative additive set function such that $\gamma\leq \alpha$ and $\gamma\leq \beta$.

Our claim is to prove that the Lebesgue decomposition of $\beta$ with respect to $\alpha$ can also be derived from that of the induced positive operators. To this aim we note first that singularity of $A$ and $B$ obviously implies the singularity of $\alpha$ and $\beta$. It is less obvious that $A$-absolute continuity of $B$ implies the $\alpha$-absolute continuity of $\beta$ (cf. also \cite{Tarcsay_RN}*{Lemma 3.1}). To see this consider a sequence $\seq R$ of $\ring$ such that $\alpha(R_n)\to0$. Clearly,
\begin{equation*}
    \sipa{J_A^*\chi^{}_{R_n}}{J_A^*\chi^{}_{R_n}}=\dual{A\chi^{}_{R_n}}{\chi^{}_{R_n}}\to0.
\end{equation*}
Since $\sipb{J_B^*\chi^{}_{R_n}}{J_B^*\chi^{}_{R_n}}\leq \beta(X)$, the sequence $(J_B^*\chi^{}_{R_n})_{\nen}$ is bounded in $\hilb$, and for every $\xi\in\dom \widehat{B}^*$,
\begin{equation*}
    \sipb{J_B^*\chi^{}_{R_n}}{\xi}=\sipa{J_A^*\chi^{}_{R_n}}{\widehat  B^*\xi}\to0.
\end{equation*}
Consequently, $J_B^*\chi^{}_{R_n}\to 0$ weakly in $\hilb$, and hence $B\chi^{}_{R_n}\to0$ in $\bar{\step}^*$ with respect to the weak-* topology $\sigma(\bar{\step}^*,\step)$. This implies that 
\begin{equation*}
    \beta(R_n)=\dual{B\chi^{}_{R_n}}{1}\to0,
\end{equation*}
hence  $\beta\ll\alpha$.

\begin{theorem}\label{T:finiteaddLeb}
Let  $\alpha,\beta:\ring\to\dupR_+ $ be nonnegative additive set functions. There exist two nonnegative additive set functions $\beta_a,\beta_s$   such that $\beta=\beta_a+\beta_s$, where $\beta_a$  is $\alpha$-absolutely continuous and $\beta_s$ is $\alpha$-singular.
\end{theorem}
\begin{proof}
Consider the $A$-Lebesgue-decomposition $B=B_a+B_s$ of the corresponding induced operators. According to the above observation it suffices to show that $B_a$ (and hence also $B_s$) is induced by an additive set function $\beta_a$ (respectively, $\beta_s$). By Proposition \ref{P:8.4}, this will be done if we prove that 
\begin{equation}\label{E:Baphi}
    \dual{B_a\phi}{\phi}=\dual{B_a\abs\phi}{\abs\phi},\qquad \phi\in\step.
\end{equation}
Set
\begin{equation*}
    f(\phi)\coloneqq\int \phi\,d\alpha,\qquad g(\phi)\coloneqq \int\phi\,d\beta,\qquad \phi\in\step,
\end{equation*}
so that $f,g$ are representable functionals on $\step$. By Theorem \ref{T:representLebesgue}, $g$ splits into $f$-absolutely continuous and $f$-singular parts $g_a$, $g_s$ respectively. By \eqref{E:8.4},
\begin{equation*}
    \dual{B_a\phi}{\phi}=g_a(\abs{\phi}^2),\qquad \phi\in\step.
\end{equation*}
This obviously gives \eqref{E:Baphi}.
\end{proof}
Finally, assume that $\ring$ is a $\sigma$-algebra and  $\mu,\nu$ are finite measures on $\ring$. By Theorem \ref{T:finiteaddLeb}  there exist two nonnegative (finitely) additive set functions $\nu_a,\nu_s$ such that $\nu=\nu_a+\nu_s$ where $\nu_a\ll \mu$ and $\nu_s\perp \mu$. Note that both functions are dominated by the measure $\nu$, hence $\nu_a,\nu_s$ are forced to be $\sigma$-additive. This fact leads us the Lebesgue decomposition of measures:
\begin{corollary}
If $\mu,\nu$ are finite measures on a $\sigma$-algebra $\ring$ then there exist two measures $\nu_a,\nu_s$ such that $\nu=\nu_a+\nu_s$, where $\nu_a$  is $\mu$-absolutely continuous and $\nu_s$ is $\mu$-singular.
\end{corollary}


\begin{thebibliography}{10}
\bibitem{AndersonDuffin}
W. N. Anderson and R. J. Duffin, 
\newblock Series and parallel addition of matrices, \emph{J. Math.
Anal. Appl.}, \textbf{26} (1969), 576--594.

\bibitem{Ando}
T. Ando,
\newblock Lebesgue-type decomposition of positive operators,
{\em Acta Sci. Math. (Szeged)}, \textbf{38} (1976), 253--260.

\bibitem{Antezena}
J. Antezena, G. Corach, D. Stojanoff, Bilateral shorted operators and parallel sums, \emph{Linear Algebra and Appl.}, \textbf{414} (2006), 570--588.

\bibitem{Arens}
R.~Arens,
\newblock Operational calculus of linear relations,
\newblock {\em Pacific J. Math.}, \textbf{11} (1961), 9--23.

\bibitem{Arlinskii}
Yu. M. Arlinskii,
\newblock On the mappings connected with parallel addition of nonnegative operators,
\newblock{\em Positivity}, \textbf{21} (2017), 299--327.

\bibitem{Barnes}
B. Barnes,
\newblock Majorization, range inclusion and factorization for bounded linear operators, 
\newblock {\em Proc. Amer. Math. Soc.}, \textbf{133} (2005), 155--162.

\bibitem{trapani}
S. di Bella and C. Trapani,
Some representation theorems for sesquilinear forms, \emph{J. Math. Anal. Appl.},  \textbf{451} (2017), 64--83.


\bibitem{rosario1}
R. Corso, A Lebesgue-type decomposition for non-positive sesquilinear forms, \emph{Annali di Matematica Pura ed Applicata}, 2018, \emph{online first.}

\bibitem{rosario2}
R. Corso, Kato's second type representation theorem for solvable sesquilinear forms, \emph{ J. Math. Anal. Appl.}, \textbf{426} (2018), 982--998.

\bibitem{rosario3}
R. Corso and C. Trapani, Representation theorems for solvable sesquilinear forms, \emph{Integral Equations Operator Theory}, \textbf{89} (2017), 43--68.


\bibitem{Djikic}
M. Djiki\'c, Extensions of the Fill--Fishkind formula and the
infimum - parallel sum relation, \emph{Linear and Multilinear Algebra}, \textbf{64} (2016),  2335--2349.



\bibitem{Douglas}
R.~G. Douglas,
\newblock On majorization, factorization, and range inclusion of operators on
  {H}ilbert space,
\newblock {\em Proc. Amer. Math. Soc.}, \textbf{17} (1966), 413--415.



\bibitem{tES}
A. F. M. ter Elst and M. Sauter, The regular part of second-order differential sectorial forms
with lower-order terms, \emph{J. Evol. Equ.} \textbf{13} (2013), 737--749.

\bibitem{Embry}
M. R. Embry,
\newblock Factorization of operators on Banach space,
\newblock {\em Proc. Amer. Math. Soc.}, \textbf{38} (1973), 587--590.

\bibitem{fillmore}
P.~A. Fillmore and J.~P. Williams,
\newblock On operator ranges,
\newblock {\em Advances in Math.}, \textbf{7} (1971), 254--281.


\bibitem{Gheondea1}
A. Gheondea and A. S. Kavruk, Absolute continuity for operator valued completely positive maps on $C^*$-algebras, \emph{Journal of Mathematical Physics} \textbf{50} (2009), 022102.


\bibitem{Gudder}
S. P. Gudder,
\newblock A Radon-Nikodym theorem for $^*$-algebras,
\newblock {\em  Pacific J. Math.,} \textbf{80} (1979), 141--149.

\bibitem{Hassi2007}
S.~Hassi, Z.~Sebesty{\'e}n, H.~S.~V. de~Snoo, and F.~H. Szafraniec,
\newblock A canonical decomposition for linear operators and linear relations,
\newblock {\em Acta Math. Hungar.}, \textbf{115} (2007), 281--307.

\bibitem{Hassi2009}
S.~Hassi, Z.~Sebesty{\'e}n, and H.~de~Snoo,
\newblock Lebesgue type decompositions for nonnegative forms,
\newblock {\em J. Funct. Anal.}, \textbf{257} (2009), 3858--3894.

\bibitem{Hassilebesgue2018}
S.~Hassi, Z.~Sebesty{\'e}n, and H. S. V.~de~Snoo, Lebesgue type decompositions for linear relations and Ando's uniqueness criterion, {\em Acta Sci. Math. (Szeged)}, \textbf{84} (2018), 465--507.


\bibitem{Hassi2009a}
S.~Hassi, H.S.V.~de~Snoo, and F. H. Szafraniec,
\newblock Componentwise and canonical decompositions of linear relations,
\newblock \emph{Dissertationes Mathematicae}, \textbf{465} (2009), 59pp.


\bibitem{Inoue}
A. Inoue, A Radon-Nikodym theorem for positive linear functionals on *-algebras, \emph{Journal of Operator Theory} \textbf{100} (1983), 77--86.

\bibitem{Kosaki}
H. Kosaki,
\newblock Lebesgue decomposition of states on a von Neumann algebra,
\newblock {\em American J. Math.}, \textbf{107} (1985), 697--735.

\bibitem{Kosaki2018}
H. Kosaki,
\newblock Parallel sum of unbounded positive operators, \emph{Kyushu Journal of Mathematics}, \textbf{71} (2017), 387--405.

\bibitem{Kosaki2019}
H. Kosaki,
\newblock Absolute continuity for unbounded positive self-adjoint operators, \emph{Kyushu Journal of Mathematics}, \textbf{72} (2018), 407--421.


\bibitem{palmer}
T. W. Palmer,
\newblock  {\em Banach Algebras and the General Theory of $^*$-Algebras II},
\newblock Cambridge University Press (2001).

\bibitem{Pekarev}
E. L. Pekarev and Yu. L. Shmulyan, 
\newblock Parallel addition and parallel subtraction of
operators, \emph{Izv. Akad. Nauk. SSSR}, \textbf{40} (1976), 366--387.

\bibitem{Schaefer}
H. H. Schaefer, {Topological vector spaces}, Third printing corrected. Graduate Texts in Mathematics, Vol. 3. \emph{Springer-Verlag}, New York-Berlin, 1971.

\bibitem{Sebestyen83}
Z.~Sebesty{\'e}n,
\newblock On ranges of adjoint operators in Hilbert space,
\newblock {\em Acta Sci. Math. (Szeged)}, \textbf{46} (1983), 295--298.


\bibitem{Simon}
B. Simon, A canonical decomposition for quadratic forms with applications to monotone convergence
theorems, \emph{J. Funct. Anal.} \textbf{28}  (1978), 377--385.

\bibitem{STT}
\newblock{Z.~Sebesty\'en, Zs.~Tarcsay and T.~Titkos},
\newblock{Lebesgue decomposition theorems},
\newblock {\em Acta Sci. Math. (Szeged)}, \textbf{79} (2013), 219--233.

 \bibitem{TitokRN}
Z.~Sebesty{\'e}n  and T. Titkos,
A Radon-Nikodym type theorem for forms, \emph{Positivity}, \textbf{17} (2013),
863--873.

\bibitem{Smulian}
\newblock{Yu. L. Shmul'yan},
\newblock{ Two-sided division in the ring of operators},
{\em Math. Notes}, \newblock \textbf{1} (1967), 400--403.

\bibitem{Szucs2012}
\newblock{Zs.~Sz\H{u}cs},
\newblock The singularity of positive linear functionals,
\newblock {\em Acta Math. Hungar.}, \textbf{136} (2012), 138--155.


\bibitem{Szucs_abs}
Zs. Sz\H ucs,
\newblock Absolute continuity of positive linear functionals,
\newblock {\em Banach J. of Math. Anal.}, \textbf{66} (2015), 201--247.


\bibitem{Tarcsay_Leb}
\newblock{Zs.~Tarcsay},
\newblock{ Lebesgue-type decomposition of positive operators},
{\em Positivity}, \newblock \textbf{17} (2013), 803--817.


\bibitem{Tarcsay_RN}
\newblock{Zs.~Tarcsay},
\newblock{Radon--Nikodym theorems for nonnegative forms, measures and representable functionals,}
{\em Complex Analysis and Operator Theory}, 
\newblock \textbf{10} (2016), 479--494;

\bibitem{Tarcsay_repr}
\newblock{Zs.~Tarcsay},
\newblock{Lebesgue decomposition for representable functionals on $^*$-al\-geb\-ras,}
{\em Glasgow Math. Journal}, \textbf{58} (2016), 491--501;



\bibitem{Tarcsay_parallel}
Zs. Tarcsay,
\newblock On the parallel sum of positive operators, forms, and functionals,
\newblock {\em Acta Math. Hungar.,} \textbf{147} (2015), 408--426.

\bibitem{TZSTT-Glasgow}
Zs. Tarcsay and T. Titkos,
\newblock On the order structure of representable functionals,
\emph{Glasgow Math. Journal}, \textbf{60} (2018), 289--305.

\bibitem{TZS-TT:KreinNeumannADP}
Zs. Tarcsay and T. Titkos,
\newblock Operators on anti-dual pairs: Generalized Krein--von Neumann extension, 2019; arXiv:1810.02619 [math.FA] 

\bibitem{Titkos_content}
T. Titkos,
\newblock Lebesgue decomposition of contents via nonnegative forms,
\newblock {\em Acta Math. Hungar.,} \textbf{140} (2013), 151--161.


\bibitem{Titkos_means}
T. Titkos,
\newblock On means of nonnegative sesquilinear forms,
\newblock {\em Acta Math. Hungar.,} \textbf{143} (2014), 515--533.

\bibitem{Titkos_howto}
T. Titkos,
\newblock How to Make Equivalent Measures?,
\newblock {\em Amer. Math. Monthly,}   \textbf{122} (2015), 812--812.

\bibitem{Arlinskii-T}
T. Titkos,
\newblock{Arlinskii's Iteration and its Applications
},
\newblock{\em Proceedings of the Edinburgh Mathematical Society},
\textbf{62} (2019), 125--133.

\bibitem{vogt}
H. Vogt, The regular part of symmetric forms associated with  second-order elliptic differential expressions, \emph{Bull. London Math. Soc.}, \textbf{41} (2009) 441--444.
\end{thebibliography}
\end{document}